\documentclass[10pt,a4,reqno]{amsart}
\usepackage{mathptmx,amsmath,mathbbol,amssymb,mathbbol,lineno,amsthm}
\usepackage{epic}
\usepackage{pstricks}

\newtheorem{theorem}{Theorem}[section]
\newtheorem{proposition}[theorem]{Proposition}
\newtheorem{lemma}[theorem]{Lemma}
\newtheorem{corollary}[theorem]{Corollary}
\newtheorem{definition}[theorem]{Definition}
\newtheorem{example}[theorem]{Example}
\newtheorem{remark}[theorem]{Remark}

\newtheorem{assumption}[theorem]{Assumption}

\numberwithin{equation}{section}
\input{tcilatex}

\begin{document}
\title[Quasi-linear elliptic boundary value problems at resonance]{
Nonlinear elliptic boundary value problems at resonance with nonlinear
Wentzell boundary conditions}
\author{Ciprian G. Gal}
\address{C. G. Gal, Department of Mathematics, University of
Missouri,Columbia, MO 65211 (USA)}
\email{galc@missouri.edu}
\author{Mahamadi Warma}
\address{M.~Warma, University of Puerto Rico, Department of Mathematics (Rio
Piedras Campus), PO Box 70377 San Juan PR 00936-8377 (USA)}
\email{mjwarma@gmail.com, warma@uprrp.edu}
\dedicatory{Dedicated to the 70th birthday of Jerome A. Goldstein}
\date{\today }

\begin{abstract}
In the first part of the article, we give necessary and sufficient
conditions for the solvability of a class of nonlinear elliptic boundary
value problems with nonlinear boundary conditions involving the $q$%
-Laplace-Beltrami operator. In the second part, we give some additional
results on existence and uniqueness and we study the regularity of the weak
solutions for these classes of nonlinear problems. More precisely, we show
some global a priori estimates for these weak solutions in an $L^{\infty }$%
-setting.
\end{abstract}

\maketitle

\section{Introduction}

Let $\Omega \subset \mathbf{R}^{N},$ $N\geq 1,$ be a bounded domain with a
Lipschitz boundary $\partial \Omega $ and consider the following nonlinear
boundary value problem with nonlinear second order boundary conditions: 
\begin{equation}
\begin{cases}
-\Delta _{p}u+\alpha _{1}\left( u\right) =f\left( x\right) , & \text{ in }%
\Omega , \\ 
&  \\ 
b\left( x\right) \left\vert \nabla u\right\vert ^{p-2}\partial _{\mathbf{n}%
}u-\rho b\left( x\right) \Delta _{q,\Gamma }u+\alpha _{2}\left( u\right)
=g\left( x\right) , & \text{ on }\partial \Omega ,%
\end{cases}
\label{1.1}
\end{equation}%
where $b\in L^{\infty }\left( \partial \Omega \right) ,$ $b(x)\geq b_{0}>0,$
for some constant $b_{0}$, $\rho $ is either $0$ or $1,$ and $\alpha _{1},$ $%
\alpha _{2}\in C\left( \mathbb{R},\mathbb{R}\right) $ are monotone
nondecreasing functions such that $\alpha _{i}\left( 0\right) =0$. Moreover, 
$\Delta _{p}u=\mbox{div}(\left\vert \nabla u\right\vert ^{p-2}\nabla u)$ is
the $p$-Laplace operator, $p\in \left( 1,+\infty \right) $ and $f\in
L^{2}\left( \Omega ,dx\right) ,$ $g\in L^{2}(\partial \Omega ,\sigma )$ are
given real-valued functions. Here, $dx$ denotes the usual $N$-dimensional
Lebesgue measure in $\Omega $ and $\sigma $ denotes the restriction to $%
\partial \Omega $ of the $(N-1)$-dimensional Hausdorff measure. Recall that $%
\sigma $ coincides with the usual Lebesgue surface measure since $\Omega $
has a Lipschitz boundary, and $\partial _{\mathbf{n}}u$ denotes the normal
derivative of $u$ in direction of the outer normal vector $\overrightarrow{%
\mathbf{n}}$. Furthermore, $\Delta _{q,\Gamma }$ is defined as the
generalized $q$-Laplace-Beltrami operator on $\partial \Omega ,$ that is, $%
\Delta _{q,\Gamma }u=\mbox{div}_{\Gamma }(\left\vert \nabla _{\Gamma
}u\right\vert ^{q-2}\nabla _{\Gamma }u),$ $q\in \left( 1,+\infty \right) $.
In particular, $\Delta _{2}=\Delta $ and $\Delta _{2,\Gamma }=\Delta
_{\Gamma }$ become the well-known Laplace and Laplace-Beltrami operators on $%
\Omega $ and $\partial \Omega ,$ respectively. Here, for any real valued
function $v,$%
\begin{equation*}
\mbox{div}_{\Gamma }v=\sum_{i=1}^{N-1}\partial _{\tau _{i}}v,
\end{equation*}%
where $\partial _{\tau _{i}}v$ denotes the directional derivative of $v$
along the tangential directions $\tau _{i}$ at each point on the boundary,
whereas $\nabla _{\Gamma }v=\left( \partial _{\tau _{1}}v,...,\partial
_{\tau _{N-1}}v\right) $ denotes the tangential gradient at $\partial \Omega 
$. It is worth mentioning again that when $\rho =0$ in (\ref{1.1}), the
boundary conditions are of lower order than the order of the $p$ -Laplace
operator, while for $\rho =1,$ we deal with boundary conditions which have
the same differential order as the operator acting in the domain $\Omega .$
Such boundary conditions arise in many applications, such as
phase-transition phenomena (see, e.g., \cite{GM, GM2} and the references
therein) and have been studied by several authors (see, e.g., \cite{ADN,
GGGR, H, Pe, V}).\newline

In a recent paper \cite{GGGR}, the authors have formulated necessary and
sufficient conditions for the solvability of (\ref{1.1}) when $p=q=2,$ by
establishing a sort of "nonlinear Fredholm alternative" for such elliptic
boundary value problems. We shall now state their main result. Defining two
real parameters $\lambda _{1},$ $\lambda _{2}\in \mathbb{R}_{+}$ by 
\begin{equation}
\lambda _{1}=\int_{\Omega }dx,\text{ }\lambda _{2}=\int_{\partial \Omega }%
\frac{d\sigma }{b},  \label{1.2}
\end{equation}%
this result reads that a necessary condition for the existence of a weak
solution of (\ref{1.1}) is that 
\begin{equation}
\int_{\Omega }f\left( x\right) dx+\int_{\partial \Omega }g\left( x\right) 
\frac{d\sigma }{b\left( x\right) }\in \left( \lambda _{1}\mathcal{R}\left(
\alpha _{1}\right) +\lambda _{2}\mathcal{R}\left( \alpha _{2}\right) \right)
,  \label{1.3}
\end{equation}%
while a sufficient condition is 
\begin{equation}
\int_{\Omega }f\left( x\right) dx+\int_{\partial \Omega }g\left( x\right) 
\frac{d\sigma }{b\left( x\right) }\in \mbox{int}\left( \lambda _{1}\mathcal{R%
}\left( \alpha _{1}\right) +\lambda _{2}\mathcal{R}\left( \alpha _{2}\right)
\right) ,  \label{1.4}
\end{equation}%
where $\mathcal{R}(\alpha _{j})$ denotes the range of $\alpha _{j}$, $j=1,2$
and $\mbox{int}(G)$ denotes the interior of the set $G$.

Relation (\ref{1.3}) turns out to be both necessary and sufficient if either
of the sets $\mathcal{R}\left( \alpha _{1}\right) $ or $\mathcal{R}\left(
\alpha _{2}\right) $ is an open interval. This particular result was
established in \cite[Theorem 3]{GGGR}, by employing methods from convex
analysis involving subdifferentials of convex, lower semicontinuous
functionals on suitable Hilbert spaces. As an application of our results, we
can consider the following boundary value problem%
\begin{equation}
\left\{ 
\begin{array}{ll}
-\Delta u+\alpha _{1}\left( u\right) =f\left( x\right) , & \text{in }\Omega ,
\\ 
b\left( x\right) \partial _{n}u=g\left( x\right) , & \text{on }\partial
\Omega ,%
\end{array}%
\right.  \label{1.5}
\end{equation}%
which is only a special case of (\ref{1.1}) (i.e., $\rho =0,$ $\alpha
_{2}\equiv 0$ and $p=2$). According to \cite[Theorem 3]{GGGR} (see also %
\eqref{1.4}), this problem has a weak solution if 
\begin{equation}
\int_{\Omega }f\left( x\right) dx+\int_{\partial \Omega }g\left( x\right) 
\frac{d\sigma }{b\left( x\right) }\in \mbox{int}\left( \lambda _{1}\mathcal{R%
}\left( \alpha _{1}\right) \right) ,  \label{1.6}
\end{equation}%
which yields the result of Landesman and Lazer \cite{LL} for $g\equiv 0$.
This last condition is both necessary and sufficient when the interval $%
\mathcal{R}\left( \alpha _{1}\right) $ is open. This was put into an
abstract context and significantly extended by Brezis and Haraux \cite{BH}.
Their work was much further extended by Brezis and Nirenberg \cite{BN}. The
goal of the present article is comparable to that of \cite{GGGR} since we
want to establish similar conditions to (\ref{1.4}) and (\ref{1.6})\ for the
existence of solutions to (\ref{1.1}) when $p,q\neq 2,$ with main emphasis
on the generality of the boundary conditions.

Recall that $\lambda _{1}$ and $\lambda _{2}$ are given by (\ref{1.2}). Let $%
\mathbb{I}$ be the interval $\lambda _{1}\mathcal{R}\left( \alpha
_{1}\right) +\lambda _{2}\mathcal{R}\left( \alpha _{2}\right) .$ Our first
main result is as follows (see Section 4 also).

\begin{theorem}
\label{main*}Let $\alpha _{j}:\mathbb{R}\rightarrow \mathbb{R}$ $(j=1,2)$ be
odd, monotone nondecreasing, continuous function such that $\alpha _{j}(0)=0$%
. Assume that the functions $\Lambda _{j}(t):=\int_{0}^{|t|}\alpha _{j}(s)ds$
satisfy%
\begin{equation}
\Lambda _{j}(2t)\leq C_{j}\Lambda _{j}(t),\;\mbox{ for all }\;t\in \mathbb{R}%
,
\end{equation}%
for some constants $C_{j}>1$, $j=1,2$. If $u$ is a weak solution of (\ref%
{1.1}) (in the sense of Definition \ref{def-weak-sol} below), then%
\begin{equation}
\int_{\Omega }f\left( x\right) dx+\int_{\partial \Omega }g\left( x\right) 
\frac{d\sigma }{b\left( x\right) }\in \mathbb{I}.  \label{1.8}
\end{equation}%
Conversely, if%
\begin{equation}
\int_{\Omega }f\left( x\right) dx+\int_{\partial \Omega }g\left( x\right) 
\frac{d\sigma }{b\left( x\right) }\in int\left( \mathbb{I}\right) ,
\label{1.8bis}
\end{equation}%
then (\ref{1.1}) has a weak solution.
\end{theorem}

Our second main result of the paper deals with a modified version of (\ref%
{1.1}) which is obtained by replacing the functions $\alpha _{1}\left(
s\right) ,$ $\alpha _{2}\left( s\right) $ in (\ref{1.1}) by $\overline{%
\alpha }_{1}\left( s\right) +\left\vert s\right\vert ^{p-2}s$ and $\overline{%
\alpha }_{2}\left( s\right) +\rho b\left\vert u\right\vert ^{q-2}u$,
respectively, and also allowing $\overline{\alpha }_{1},$ $\overline{\alpha }%
_{2}$ to depend on $x\in \overline{\Omega }$. Under additional assumptions
on $\overline{\alpha }_{1},\overline{\alpha }_{2}$ and under higher
integrability properties for the data $\left( f,g\right) $, the next theorem
provides us with conditions for unique solvability results for solutions to
such boundary value problems. Then, we obtain some regularity results for
these solutions. In addition to these results, the continuous dependence of
the solution to (\ref{1.1}) with respect to the data $\left( f,g\right) $
can be also established. In particular, we prove the following

\begin{theorem}
\label{main*2}Let all the assumptions of Theorem \ref{main*} be satisfied
for the functions $\overline{\alpha }_{1},$ $\overline{\alpha }_{2}$.
Moreover, for each $j=1,2$, assume that $\overline{\alpha }_{j}\left(
t\right) /t\rightarrow 0,$ as $t\rightarrow 0$ and $\overline{\alpha }%
_{j}\left( t\right) /t\rightarrow \infty ,$ as $t\rightarrow \infty $,
respectively.

(a) Then, for every $\left( f,g\right) \in L^{p_{1}}(\Omega )\times
L^{q_{1}}(\partial \Omega )$ with%
\begin{equation*}
p_{1}>\max \left\{ 1,\frac{N}{p}\right\} ,\text{ }q_{1}>\left\{ 
\begin{array}{ll}
\max \left\{ 1,\frac{N-1}{p-1}\right\} , & \text{if }\rho \in \left\{
0,1\right\} , \\ 
\max \left\{ 1,\frac{N-1}{p}\right\} , & \text{if }\rho =1\text{ and }p=q,%
\end{array}%
\right.
\end{equation*}
there exists a unique weak solution to problem (\ref{1.1}) (in the sense of
Definition \ref{def-form} below) which is bounded.

(b) Let $\overline{\alpha }_{j},$ $j=1,2,$ be such that%
\begin{equation*}
c_{j}\left\vert \overline{\alpha }_{j}(\xi -\eta )\right\vert \leq
\left\vert \overline{\alpha }_{j}(\xi )-\overline{\alpha }_{j}(\eta
)\right\vert ,\text{ }\mbox{ for
all }\;\xi ,\eta \in \mathbb{R},
\end{equation*}%
for some constants $c_{j}\in (0,1]$. Then, the weak (bounded)\ solution of
problem (\ref{1.1}) depends continuously on the data $\left( f,g\right) $.
Precisely, let us indicate by $u_{F_{j}}$ the unique solution corresponding
to the data $F_{j}:=\left( f_{j},g_{j}\right) \in L^{p_{1}}(\Omega )\times
L^{q_{1}}(\partial \Omega ),$ for each $j=1,2$. Then, the following estimate
holds:%
\begin{equation*}
\Vert u_{F_{1}}-u_{F_{2}}\Vert _{L^{\infty }(\Omega )}+\Vert
u_{F_{1}}-u_{F_{2}}\Vert _{L^{\infty }(\partial \Omega )}\leq Q\left( \Vert
f_{1}-f_{2}\Vert _{L^{p_{1}}(\Omega )},\Vert g_{1}-g_{2}\Vert
_{L^{q_{1}}(\partial \Omega )}\right) ,
\end{equation*}%
for some nonnegative function $Q:\mathbb{R}_{+}^{2}\rightarrow \mathbb{R}%
_{+} $, $Q\left( 0,0\right) =0$, which can be computed explicitly.
\end{theorem}

We organize the paper as follows. In Section \ref{preli}, we introduce some
notations and recall some well-known results about Sobolev spaces, maximal
monotone operators and Orlicz type spaces which will be needed throughout
the article. In Section \ref{aux}, we show that the subdifferential of a
suitable functional associated with problem (\ref{1.1}) satisfies a sort of
"quasilinear" version of the Fredholm alternative (cf. Theorem \ref{QFA}),
which is needed in order to obtain the result in Theorem \ref{main*}.
Finally, in Sections \ref{main} and \ref{priori}, we provide detailed proofs
of Theorem \ref{main*} and Theorem \ref{main*2}. We also illustrate the
application of these results with some examples.

\section{Preliminaries and notations}

\label{preli}

In this section we put together some well-known results on nonlinear forms,
maximal monotone operators and Sobolev spaces. For more details on maximal
monotone operators, we refer to the monographs \cite{BC,Br73,Min1,Min2,Scho}%
. We will also introduce some notations.

\subsection{Maximal monotone operators}

Let $H$ be a real Hilbert space with scalar product $(\cdot,\cdot)_H$.

\begin{definition}
Let $A:\;D(A)\subset H\to H$ be a closed (nonlinear) operator. The operator $%
A$ is said to be:

\begin{enumerate}
\item[(i)] \emph{monotone}, if for all $u,v\in D(A)$ one has%
\begin{equation*}
(Au-Av,u-v)_{H}\geq 0.
\end{equation*}

\item[(ii)] \emph{maximal monotone}, if it is monotone and the operator $I+A$
is invertible.
\end{enumerate}
\end{definition}

Next, let $V$ be a real reflexive Banach space which is densely and
continuously embedded into the real Hilbert space $H$, and let $V^{\prime}$
be its dual space such that $V\hookrightarrow H\hookrightarrow V^{\prime}$.

\begin{definition}
\label{def-11} Let $\mathcal{A}:\;V\times V\to \mathbb{R}$ be a continuous
map.

\begin{enumerate}
\item The map $\mathcal{A}:\;V\times V\to \mathbb{R}$ is called a \emph{%
nonlinear form} on $H$ if for all $u\in V$ one has $\mathcal{A}(u,\cdot)\in
V^{\prime}$, that is, if $\mathcal{A}$ is linear and bounded in the second
variable.

\item The nonlinear form $\mathcal{A}:\;V\times V\to \mathbb{R}$ is said to
be:

\begin{enumerate}
\item[(i)] \emph{monotone} if %\begin{equation*}
$\mathcal{A}(u,u-v)-\mathcal{A}(v,u-v)\ge 0\;\mbox{ for all }\;u,v\in V$; 
%\end{equation*}

\item[(ii)] \emph{hemicontinuous} if %\begin{equation*}
$\displaystyle\lim_{t\downarrow 0} \mathcal{A}(u+tv,w)=\mathcal{A}%
(u,w),\;\;\forall\;u,v,w\in V$ ; %\end{equation*}

\item[(iii)] \emph{coercive}, if %\begin{equation*}
$\displaystyle\lim_{\|v\|_V\to+\infty}\frac{\mathcal{A}(v,v)}{\|v\|_V}%
=+\infty$. %\end{equation*}
\end{enumerate}
\end{enumerate}
\end{definition}

Now, let $\varphi :\;H\rightarrow (-\infty ,+\infty ]$ be a proper, convex,
lower semicontinuous functional with effective domain 
\begin{equation*}
D(\varphi ):=\{u\in H:\;\varphi (u)<\infty \}.
\end{equation*}%
The subdifferential $\partial \varphi $ of the functional $\varphi $ is
defined by 
\begin{equation*}
\begin{cases}
\displaystyle D(\partial \varphi ) & :=\{u\in D(\varphi ):\;\exists \;w\in
H\;\forall \;v\in D(\varphi ):\;\varphi (v)-\varphi (u)\geq (w,v-u)_{H}\};
\\ 
&  \\ 
\displaystyle\partial \varphi (u) & :=\{w\in H:\;\forall \;v\in D(\varphi
):\;\varphi (v)-\varphi (u)\geq (w,v-u)_{H}\}.%
\end{cases}%
\end{equation*}
By a classical result of Minty \cite{Min1} (see also \cite{Br73,Min2}), $%
\partial\varphi$ is a maximal monotone operator.

\subsection{Functional setup}

Let $\Omega \subset \mathbb{R}^{N}$ be a bounded domain with a Lipschitz
boundary $\partial \Omega $. For $1<p<\infty $, we let $W^{1,p}(\Omega )$ be
the first order Sobolev space, that is, 
\begin{equation*}
W^{1,p}(\Omega )=\{u\in L^{p}(\Omega ):\;\nabla u\in (L^{p}(\Omega ))^{N}\}.
\end{equation*}%
Then $W^{1,p}(\Omega ),$ endowed with the norm%
\begin{equation*}
\Vert u\Vert _{W^{1,p}(\Omega )}:=\left( \Vert u\Vert _{\Omega ,p}^{p}+\Vert
\nabla u\Vert _{\Omega ,p}^{p}\right) ^{1/p}
\end{equation*}%
is a Banach space, where we have set%
\begin{equation*}
\Vert u\Vert _{\Omega ,p}^{p}:=\int_{\Omega }|u|^{p}\;dx.
\end{equation*}%
Since $\Omega $ has a Lipschitz boundary, it is well-known that there exists
a constant $C>0$ such that%
\begin{equation}
\Vert u\Vert _{\Omega ,p_{s}}\leq C\Vert u\Vert _{W^{1,p}(\Omega )},\;\text{%
for all }u\in W^{1,p}(\Omega ),  \label{sob-ine1}
\end{equation}%
where $p_{s}=\frac{pN}{N-p}$ if $p<N,$ and $1\leq p_{s}<\infty $ if $N=p$.
Moreover the trace operator $\mbox{Tr}(u):=u_{|_{\partial \Omega }}$
initially defined for $u\in C^{1}(\bar{\Omega})$ has an extension to a
bounded linear operator from $W^{1,p}(\Omega )$ into $L^{q_{s}}(\partial
\Omega )$ where $q_{s}:=\frac{p(N-1)}{N-p}$ if $p<N$, and $1\leq
q_{s}<\infty $ if $N=p$. Hence, there is a constant $C>0$ such that 
\begin{equation}
\Vert u\Vert _{\partial \Omega ,q_{s}}\leq C\Vert u\Vert _{W^{1,p}(\Omega
)},\;\text{for all}\;u\in W^{1,p}(\Omega ).  \label{trace}
\end{equation}%
Throughout the remainder of this article, for $1<p<N$, we let 
\begin{equation}
p_{s}:=\frac{pN}{N-p}\;\mbox{ and }\;q_{s}:=\frac{p(N-1)}{N-p}.
\label{ps-qs}
\end{equation}

If $p>N$, one has that 
\begin{equation}
W^{1,p}(\Omega )\hookrightarrow C^{0,1-\frac{N}{p}}(\bar{\Omega}),
\label{cont-inj}
\end{equation}%
that is, the space $W^{1,p}(\Omega )$ is continuously embedded into $C^{0,1-%
\frac{N}{p}}(\bar{\Omega})$. For more details, we refer to \cite[Theorem 4.7]%
{necas} (see also \cite[Chapter 4]{MP}).

For $1<q<\infty $, we define the Sobolev space $W^{1,q}(\partial \Omega )$
to be the completion of the space $C^{1}(\partial \Omega )$ with respect to
the norm 
\begin{equation*}
\Vert u\Vert _{W^{1,q}(\partial \Omega )}:=\left( \int_{\partial \Omega
}|u|^{q}\;d\sigma +\int_{\partial \Omega }|\nabla _{\Gamma }u|^{q}\;d\sigma
\right) ^{1/q},
\end{equation*}%
where we recall that $\nabla _{\Gamma }u$ denotes the tangential gradient of
the function $u$ at the boundary $\partial \Omega $. It is also well-known
that $W^{1,q}(\partial \Omega )$ is continuously embedded into $%
L^{q_{t}}(\partial \Omega )$ where $q_{t}:=\frac{q(N-1)}{N-1-q}$ if $%
1<q<N-1, $ and $1\leq q_{t}<\infty $ if $q=N-1$. Hence, for $1<q\le N-1$,
there exists a constant $C>0$ such that%
\begin{equation}
\Vert u\Vert _{q_{t},\partial \Omega }\leq C\Vert u\Vert _{W^{1,q}(\partial
\Omega )},\text{ for all }u\in W^{1,q}(\partial \Omega ).  \label{qt}
\end{equation}

Let $\lambda _{N}$ denote the $N$-dimensional Lebesgue measure and let the
measure $\mu :=\lambda _{N}|_{\Omega }\oplus \sigma $ on $\overline{\Omega }$
be defined for every measurable set $A\subset \overline{\Omega }$ by%
\begin{equation*}
\mu (A):=\lambda _{N}(\Omega \cap A)+\sigma (A\cap \partial \Omega ).
\end{equation*}%
For $p,q\in \lbrack 1,\infty ],$ we define the Banach space 
\begin{equation*}
X^{p,q}(\overline{\Omega },\mu ):=\{F=(f,g):\;f\in L^{p}(\Omega )\mbox{ and }%
g\in L^{q}(\partial \Omega )\}
\end{equation*}%
endowed with the norm 
\begin{equation*}
\Vert F\Vert _{X^{p,q}(\overline{\Omega })}=\Vert |F\Vert |_{p,q}:=\Vert
f\Vert _{\Omega ,p}+\Vert g\Vert _{\partial \Omega ,q},
\end{equation*}%
if $1\leq p,q<\infty ,$ and 
\begin{equation*}
\Vert F\Vert _{X^{\infty ,\infty }(\overline{\Omega },\mu )}=\Vert |F\Vert
|_{\infty }:=\max \{\Vert f\Vert _{\Omega ,\infty },\Vert g\Vert _{\partial
\Omega ,\infty }\}.
\end{equation*}%
If $p=q$, we will simply denote $\Vert |F\Vert |_{p,p}=\Vert |F\Vert |_{p}$.

Identifying each function $u\in W^{1,p}(\Omega )$ with $U=(u,u|_{\partial
\Omega })$, we have that $W^{1,p}(\Omega )$ is a subspace of $X^{p,p}(%
\overline{\Omega },\mu )$.

For $1<p,q<\infty ,$ we endow%
\begin{equation*}
\mathcal{V}_{1}:=\{U:=(u,u|_{\partial \Omega }),u\in W^{1,p}(\Omega
),\;u|_{\partial \Omega }\in W^{1,q}(\partial \Omega )\}
\end{equation*}%
with the norm 
\begin{equation*}
\Vert U\Vert _{\mathcal{V}_{1}}:=\Vert u\Vert _{W^{1,p}(\Omega )}+\Vert
u\Vert _{W^{1,q}(\partial \Omega )},
\end{equation*}%
while%
\begin{equation*}
\mathcal{V}_{0}:=\{U=(u,u|_{\partial \Omega }):\;u\in W^{1,p}(\Omega )\}
\end{equation*}%
is endowed with the norm 
\begin{equation*}
\Vert U\Vert _{\mathcal{V}_{0}}:=\Vert u\Vert _{W^{1,p}(\Omega )}.
\end{equation*}%
It follows from \eqref{sob-ine1}-\eqref{trace} that $\mathcal{V}_{0}$ is
continuously embedded into $X^{p_{s},q_{s}}(\overline{\Omega },\mu ),$ with $%
p_{s}$ and $q_{s}$ given by \eqref{ps-qs}, for $1<p<N$. Moreover, by %
\eqref{sob-ine1} and \eqref{qt}, $\mathcal{V}_{1}$ is continuously embedded
into $X^{p_{s},q_{t}}(\overline{\Omega },\mu )$.

\subsection{Musielak-Orlicz type spaces}

For the convenience of the reader, we introduce the Orlicz and
Musielak-Orlicz type spaces and prove some properties of these spaces which
will be frequently used in the sequel (see Section 5).

\begin{definition}
\label{def-24} Let $(X,\Sigma,\nu)$ be a complete measure space. We call a
function $B:X\times\mathbb{R}\to[0,\infty]$ a \emph{Musielak-Orlicz function}
on $X$ if

\begin{enumerate}
\item $B(x,\cdot)$ is non-trivial, even, convex for $\nu$-a.e. $x\in X$;

\item $B(x,\cdot)$ is vanishing and continuous at $0$ for $\nu$-a.e. $x\in X$%
;

\item $B(x,\cdot )$ is left continuous on $[0,\infty );$

\item $B(\cdot ,t)$ is $\Sigma $-measurable for all $t\in \lbrack 0,\infty )$%
;

\item $\displaystyle\lim_{t\rightarrow \infty }\frac{B(x,t)}{t}=\infty $.
\end{enumerate}
\end{definition}

The \emph{complementary Musielak-Orlicz function} $\widetilde{B}$ is defined
by 
\begin{equation*}
\widetilde{B}(x,t):=\sup \{s|t|-B(x,s):s>0\}.
\end{equation*}%
It follows directly from the definition that for $t,s\geq 0$ (and hence for
all $t,s\in \mathbb{R}$) 
\begin{equation*}
st\leq B(x,t)+\widetilde{B}(x,s).
\end{equation*}

\begin{definition}
\label{Musielak-Orlicz} We say that a Musielak-Orlicz function $B$ satisfies
the $(\triangle _{\alpha }^{0})$-condition $(\alpha >1)$ if there exists a
set $X_{0}$ of $\nu $-measure zero and a constant $C_{\alpha }>1$ such that 
\begin{equation*}
B(x,\alpha t)\leq C_{\alpha }B(x,t),
\end{equation*}%
for all $t\in \mathbb{R}$ and every $x\in X\setminus X_{0}$.

We say that $B$ satisfies the $(\nabla _{2}^{0})$-condition if there is a
set $X_{0}$ of $\nu $-measure zero and a constant $c>1$ such that 
\begin{equation*}
B(x,t)\leq \frac{1}{2c}B(x,ct),
\end{equation*}%
for all $t\in \mathbb{R}$ and all $x\in X\setminus X_{0}$.
\end{definition}

\begin{definition}
\label{N-func} A function $\Phi :\mathbb{R}\rightarrow \lbrack 0,\infty )$
is called an ${\mathcal{N}}$-function if

\begin{itemize}
\item $\Phi $ is even, strictly increasing and convex;

\item $\Phi (t)=0$ if and only if $t=0$;

\item $\displaystyle\lim_{t\rightarrow 0}\frac{\Phi (t)}{t}=0$ and $%
\displaystyle\lim_{t\rightarrow \infty }\frac{\Phi (t)}{t}=\infty $.
\end{itemize}

We say that an ${\mathcal{N}}$-function $\Phi $ satisfies the $(\triangle
_{2})$-condition if there exists a constant $C_{2}>1$ such that 
\begin{equation*}
\Phi (2t)\leq C_{2}\Phi (t),\;\;\;\mbox{ for all }\;t\in \mathbb{R},
\end{equation*}%
and it satisfies the $(\nabla _{2})$-condition if there is a constant $c>1$
such that 
\begin{equation*}
\Phi (t)\leq \Phi (ct)/(2c),\;\;\mbox{ for all }\;t\in \mathbb{R}.
\end{equation*}%
For more details on ${\mathcal{N}}$-functions, we refer to the monograph of
Adams \cite[Chapter VIII]{Adam} (see also \cite[Chapter I]{RR-1}, \cite[%
Chapter I]{RR}).
\end{definition}

\begin{remark}
\label{rem-N-func} \emph{For an ${\mathcal{N}}$-function $\Phi $, we let $%
\varphi $ be its left-sided derivative. Then $\varphi $ is left continuous
on $(0,\infty )$ and nondecreasing. Let $\psi $ be given by}%
\begin{equation*}
\psi \left( s\right) :=\inf \left\{ t>0:\varphi \left( t\right) >s\right\} .
\end{equation*}%
\emph{Then 
\begin{equation*}
\Phi (t)=\int_{0}^{|t|}\varphi (s)\;ds;\qquad \Psi (t):=\int_{0}^{|t|}\psi
(s)\;ds=\sup \{|t|s-\Phi (s):s>0\}.
\end{equation*}%
As before for all $s,t\in \mathbb{R}$, 
\begin{equation}
st\leq \Phi (t)+\Psi (s).  \label{st}
\end{equation}%
Moreover, if $s=\varphi (t)$ or $t=\psi (s)$ then we have equality, that is, 
\begin{equation}
\Psi (\varphi (t))=t\varphi (t)-\Phi (t).  \label{tt}
\end{equation}%
The function $\Psi $ is called the complementary ${\mathcal{N}}$-function of 
$\Phi $. It is also known that an ${\mathcal{N}}$-function $\Phi $ satisfies
the $(\triangle _{2})$-condition if and only if 
\begin{equation}
ct\varphi (t)\leq \Phi (t)\leq t\varphi (t),  \label{delta-2}
\end{equation}%
for some constant $c\in (0,1]$ and for all $t\in \mathbb{R}$, where $\varphi 
$ is the left-sided derivative of $\Phi $.}
\end{remark}

\begin{lemma}
Let $\Phi$ be an ${\mathcal{N}}$-function which satisfies the $(\triangle_2)$%
-condition with the constant $C_2>1$ and let $\Psi$ be its complementary ${%
\mathcal{N}}$-function. Then $\Psi$ satisfies the $(\nabla_2)$-condition
with the constant $c:=2^{C_2-1}$.
\end{lemma}

\begin{proof}
We have 
\begin{equation*}
t\varphi (t)\leq \int_{t}^{2t}\varphi (s)\;ds\leq \int_{0}^{2t}\varphi
(s)\;ds=\Phi (2t)\leq C_{2}\Phi (t).
\end{equation*}%
Since $\varphi (\psi (s))\geq s$ for all $s\geq 0$ and $s/\Psi (s)$ and $%
s/(s-1)$ are decreasing, we get for $t:=\psi (s),$ that%
\begin{equation*}
\frac{s\psi (s)}{\Psi (s)}\geq \frac{\varphi (\psi (s))\psi (s)}{\Psi
(\varphi (\psi (s)))}=\frac{t\varphi (t)}{\Psi (\varphi (t))}=\frac{t\varphi
(t)}{t\varphi (t)-\Phi (t)}\geq \frac{C_{2}}{C_{2}-1}.
\end{equation*}%
Now let $c:=2^{C_{2}-1}$. Then for $t\geq 0,$%
\begin{align*}
\ln \left( \frac{\Psi (ct)}{\Psi (t)}\right) & =\int_{t}^{ct}\frac{\psi (s)}{%
\Psi (s)}\;ds\geq \int_{t}^{ct}\frac{C_{2}}{s(C_{2}-1)}\;ds \\
& =\frac{C_{2}}{C_{2}-1}\ln (c)=C_{2}\log (2)=\ln (2\cdot 2^{C_{2}-1}).
\end{align*}%
Hence, $\Psi (t)2c\leq \Psi (ct)$.
\end{proof}

\begin{corollary}
Let $B$ be a Musielak-Orlicz function such that $B(x,\cdot)$ is an ${%
\mathcal{N}}$-function for $\nu$-a.e. $x$. If $B$ satisfies the $%
(\triangle_2^0)$-condition, then $\widetilde B$ satisfies the $(\nabla_2^0)$%
-condition.
\end{corollary}

\begin{definition}
Let $B$ be a Musielak-Orlicz function. Then the Musielak-Orlicz space $%
L^{B}(X)$ associated with $B$ is defined by 
\begin{equation*}
L^{B}(X):=\{u:X\rightarrow \mathbb{R}\text{ measurable }:\rho _{B}(u/\alpha
)<\infty \text{ for some }\alpha >0\},
\end{equation*}%
where%
\begin{equation*}
\rho _{B}(v):=\int_{X}B(x,v(x))\;d\nu (x).
\end{equation*}%
On this space we consider the Luxemburg norm $\Vert \cdot \Vert _{X,B}$
defined by 
\begin{equation*}
\Vert u\Vert _{X,B}:=\inf \{\alpha >0:\rho _{B}(u/\alpha )\leq 1\}.
\end{equation*}
\end{definition}

\begin{proposition}
\label{prop:coercive} Let $B$ be a Musielak-Orlicz function which satisfies
the $(\nabla _{2}^{0})$ -condition. Then 
\begin{equation*}
\lim_{\Vert u\Vert _{X,B}\rightarrow +\infty }\frac{\rho _{B}(u)}{\Vert
u\Vert _{X,B}}=+\infty .
\end{equation*}
\end{proposition}

\begin{proof}
If $B$ satisfies the $(\nabla _{2}^{0})$-condition, then there exists a set $%
X_{0}\subset X$ of measure zero such that for every $\varepsilon >0$ there
exists $\alpha =\alpha (\varepsilon )>0$,%
\begin{equation}
B(x,\alpha t)\leq \alpha \varepsilon B(x,t),  \label{B11}
\end{equation}%
for all $t\in \mathbb{R}$ and all $x\in X\backslash X_{0}$. Let $\lambda \in
(0,\infty )$ be fixed. For $\varepsilon :=1/\lambda $ there exists $\alpha
>0 $ satisfying the above inequality. We will show that $\rho _{B}(u)\geq
\lambda \Vert u\Vert _{X,B}$ whenever $\Vert u\Vert _{X,B}>1/\alpha $.
Assume that $\Vert u\Vert _{X,B}>1/\alpha $ and let $\delta >0$ be such that 
$\alpha =(1+\delta )/\Vert u\Vert _{X,B}$. Then 
\begin{align*}
\rho _{B}(\alpha u)& =\int_{X}B(x,u(1+\delta )/\Vert u\Vert _{X,B})\;d\mu \\
& \geq (1+\delta )^{1-1/n}\int_{X}B(x,u(1+\delta )^{1/n}/\Vert u\Vert
_{X,B})\;d\mu \geq (1+\delta )^{1-1/n},
\end{align*}%
for all $n\in \mathbb{N}$. If we assume that the last inequality does not
hold, then 
\begin{equation*}
\Vert u\Vert _{X,B}/(1+\delta )\in \{\alpha >0:\rho (u/\alpha )\leq 1\},
\end{equation*}%
and this clearly contradicts the definition of $\Vert u\Vert _{X,B}$.
Therefore, we must have%
\begin{equation}
\rho _{B}(\alpha u)\geq 1+\delta =\alpha \Vert u\Vert _{X,B}.  \label{B12}
\end{equation}%
From (\ref{B11}), (\ref{B12}), we obtain%
\begin{equation*}
\rho _{B}(u)=\int_{X}B(x,u(x))\;d\mu \geq \frac{\lambda }{\alpha }%
\int_{X}B(x,\alpha u(x))\;d\mu =\frac{\lambda }{\alpha }\rho _{B}(\alpha
u)\geq \lambda \Vert u\Vert _{X,B}.
\end{equation*}

The proof is finished.
\end{proof}

\begin{corollary}
\label{cor:coercive} Let $B$ be a Musielak-Orlicz function such that $%
B(x,\cdot)$ is an ${\mathcal{N}}$-function for $\nu$-a.e. $x$. If its
complementary ${\mathcal{N}}$-function $\widetilde B$ satisfies the $%
(\triangle_2^0)$-condition, then $B$ satisfies the $(\nabla_2^0)$-condition
and 
\begin{equation*}
\lim_{\|u\|_{X,B}\to+\infty} \frac{\rho_B(u)}{\|u\|_{X,B}}=+\infty.
\end{equation*}
\end{corollary}

\subsection{Some tools}

For the reader's convenience, we report here below some useful inequalities
which will be needed in the course of investigation.

\begin{lemma}
\label{lem:ab} Let $a,b\in\mathbb{R}^N$ and $p\in(1,\infty)$. Then, there
exists a constant $C_p>0$ such that 
\begin{equation}  \label{in-ab}
\left(|a|^{p-2}a-|b|^{p-2}b\right)(a-b)\ge
C_p\left(|a|+|b|\right)^{p-2}|a-b|^2\geq 0.
\end{equation}
If $p\in[2,\infty)$, then there exists a constant $c_p\in(0,1]$ such that 
\begin{equation}  \label{ine-BW}
\left(|a|^{p-2}a-|b|^{p-2}b\right)(a-b)\ge c_p|a-b|^p.
\end{equation}
\end{lemma}

\begin{proof}
The proof of \eqref{ine-BW} is included in \cite[Lemma I.4.4]{Bene}. In
order to show \eqref{in-ab}, one only needs to show that the left hand side
is non-negative, which follows easily.
\end{proof}

The following result which is of analytic nature and whose proof can be
found in \cite[Lemma 3.11]{MS} will be useful in deriving some a priori
estimates of weak solutions of elliptic equations.

\begin{lemma}
\label{lem:fallend} Let $\psi :[k_{0},\infty )\rightarrow \mathbb{R}$ be a
non-negative, non-increasing function such that there are positive constants 
$c,\alpha $ and $\delta $ ($\delta >1$) such that 
\begin{equation*}
\psi (h)\leq c(h-k)^{-\alpha }\psi (k)^{\delta },\qquad \forall \;h>k\geq
k_{0}.
\end{equation*}%
Then $\psi (k_{0}+d)=0$ with $d=c^{1/\alpha }\psi (k_{0})^{(\delta
-1)/\alpha }2^{\delta (\delta -1)}$.\newline
\end{lemma}

\section{The Fredholm alternative}

\label{aux}

In what follows, we assume that $\Omega \subset \mathbb{R}^{N}$ is a bounded
domain with Lipschitz boundary $\partial \Omega $. Let $b\in L^{\infty
}(\partial \Omega )$ satisfy $b(x)\geq b_{0}>0$ for some constant $b_{0}$.
Let $\mathbb{X}_{2}$ be the real Hilbert space $L^{2}\left( \Omega
,dx\right) \oplus L^{2}\left( \partial \Omega ,\frac{d\sigma }{b}\right) $.
Then, it is clear that $\mathbb{X}_{2}$ is isomorphic to $X^{2,2}(\overline{%
\Omega },\lambda _{N}\oplus \sigma )$ with equivalent norms.\newline

Next, let $\rho \in \{0,1\}$ and $p,q\in (1,+\infty)$ be fixed. %with
%\begin{equation}
%\begin{cases}
%\displaystyle p\in [2N/(N+1),\infty) \; & \mbox{ if }\rho
%=0,\;\mbox{ and } \\
%&  \\
%\displaystyle p\in [2N/(N+2),\infty) \;\mbox{ and }\;q\in %
%[2(N-1)/(N+1),\infty) \; & \mbox{ if }\;\rho =1.%
%\end{cases}
%\label{p-q}
%\end{equation}%
We define the functional $\mathcal{J}_{\rho }:\;\mathbb{X}_{2}\rightarrow
\lbrack 0,+\infty ]$ by setting 
\begin{equation}
\mathcal{J}_{\rho }\left( U\right) = 
\begin{cases}
\displaystyle\frac{1}{p}\int_{\Omega }\left\vert \nabla u\right\vert ^{p}dx+%
\frac{1}{q}\int_{\partial \Omega }\rho \left\vert \nabla _{\Gamma
}u\right\vert ^{q}\;d\sigma , & \text{ if }U=\left( u,u_{\mid \partial
\Omega }\right) \in D\left( \mathcal{J}_{\rho }\right) , \\ 
+\infty , & \text{ if }U\in \mathbb{X}_{2}\diagdown D\left( \mathcal{J}%
_{\rho }\right) ,%
\end{cases}
\label{2.8}
\end{equation}%
where the effective domain is given $D(\mathcal{J}_{\rho })=\mathcal{V}%
_{\rho }\cap\mathbb{X}_2$. 
%Note that, by employing \eqref{sob-ine1}, \eqref{trace}, %
%\eqref{qt}, and using the assumptions \eqref{p-q}, it is not difficult to
%show that $D(\mathcal{J}_{\rho })$ is always a subspace of $\mathbb{X}_{2}$
%for each fixed value of $\rho \in \{0,1\}$.

Throughout the remainder of this section, we let $\displaystyle\mu :=\lambda
_{N}\oplus \frac{d\sigma }{b}$. The following result can be obtained easily. 
%byfollowing line by line the proof in \cite[Theorem 3.1]{W1}.

\begin{proposition}
\label{PCLS}The functional $\mathcal{J}_{\rho }$ defined by (\ref{2.8}) is
proper, convex and lower semicontinuous on $\mathbb{X}_{2}=X^{2,2}(\overline{%
\Omega },\mu )$.
\end{proposition}

The following result contains a computation of the subdifferential $\partial 
\mathcal{J}_{\rho }$ for the functional $\mathcal{J}_{\rho }$.

\begin{remark}
\label{cal-sub} \emph{Let $U=(u,u|_{\partial \Omega })\in D(\mathcal{J}%
_{\rho })$ and let $F:=(f,g)\in \partial \mathcal{J}_{\rho }(U)$. Then, by
definition, $F\in \mathbb{X}_{2}$ and for all $V=(v,v|_{\partial \Omega
})\in D(\mathcal{J}_{\rho })$, we have%
\begin{equation*}
\int_{\overline{\Omega }}F(V-U)\;d\mu \leq \frac{1}{p}\int_{\Omega }\bigg(%
|\nabla v|^{p}-|\nabla u|^{p}\bigg)\;dx+\frac{1}{q}\rho \int_{\Omega }\bigg(%
|\nabla _{\Gamma }v|^{q}-|\nabla _{\Gamma }u|^{q}\bigg)\;d\sigma .
\end{equation*}%
Let $W=(w,w|_{\partial \Omega })\in D(\mathcal{J}_{\rho })$, $0<t\leq 1$ and
set $V:=tW+U$ above. Dividing by $t$ and taking the limit as $t\downarrow 0$%
, we obtain that 
\begin{equation}
\int_{\overline{\Omega }}FW\;d\mu \leq \int_{\Omega }|\nabla u|^{p-2}\nabla
u\cdot \nabla w\;dx+\rho \int_{\partial \Omega }|\nabla _{\Gamma
}|^{q-2}\nabla _{\Gamma }u\cdot \nabla _{\Gamma }w\,d\sigma ,  \label{Eq}
\end{equation}%
where we recall that 
\begin{equation*}
\int_{\overline{\Omega }}F\;d\mu =\int_{\Omega }f\;dx+\int_{\partial \Omega
}g\;\frac{d\sigma }{b}.
\end{equation*}%
Choosing $w=\pm \psi $ with $\psi \in \mathcal{D}(\Omega )$ (the space of
test functions) and integrating by parts in (\ref{Eq}), we obtain%
\begin{equation*}
-\Delta _{p}u=f\;\;\mbox{ in }\;\mathcal{D}^{\prime }(\Omega )
\end{equation*}%
and 
\begin{equation*}
g=b(x)\left\vert \nabla u\right\vert ^{p-2}\partial _{n}u-\rho b\left(
x\right) \Delta _{q,\Gamma }u\;\mbox{ weakly on }\;\partial \Omega .
\end{equation*}%
Therefore, the single valued operator $\partial\mathcal{J}_{\rho }$ is given
by 
\begin{equation*}
D(\partial \mathcal{J}_{\rho })=\{U=(u,u_{\mid\partial\Omega})\in D(\mathcal{%
J}_{\rho }),\; \left( -\Delta _{p}u,b(x)\left\vert \nabla u\right\vert
^{p-2}\partial _{n}u-\rho b\left( x\right) \Delta _{q,\Gamma }u\right)\in 
\mathbb{X}_2\},
\end{equation*}
and 
\begin{equation}
\partial \mathcal{J}_{\rho }(U)=\left( -\Delta _{p}u,b(x)\left\vert \nabla
u\right\vert ^{p-2}\partial _{n}u-\rho b\left( x\right) \Delta _{q,\Gamma
}u\right) .  \label{J-rho}
\end{equation}%
\qed}
\end{remark}

Since the functional $\mathcal{J}_\rho$ is proper, convex and lower
semicontinuous, it follows that its subdifferential $\partial \mathcal{J}%
_\rho$ is a maximal monotone operator.\newline

In the following two lemmas, we establish a relation between the null space
of the operator $A_\rho:=\partial\mathcal{J}_\rho$ and its range.

\begin{lemma}
\label{lem4.5bis} Let $\mathcal{N}\left( A_{\rho }\right) $ denote the null
space of the operator $A_{\rho }$. Then 
\begin{equation*}
\mathcal{N}\left( A_{\rho }\right) =C\mathbf{1}=\left\{ C=(c,c):\;c\in 
\mathbb{R}\right\} ,
\end{equation*}%
that is, $\mathcal{N}\left( A_{\rho }\right) $ consists of all the real
constant functions on $\overline{\Omega }$.
\end{lemma}

\begin{proof}
We say that $U\in \mathcal{N}\left( A_{\rho }\right) $ if and only if (by
definition) $U=(u,u|_{\partial \Omega })$ is a weak solution of 
\begin{equation}
\begin{cases}
-\Delta _{p}u=0, & \text{ in }\Omega , \\ 
b\left( x\right) \left\vert \nabla u\right\vert ^{p-2}\partial _{n}u-\rho
b\left( x\right) \Delta _{q,\Gamma }u=0, & \text{ on }\partial \Omega .%
\end{cases}
\label{3.4}
\end{equation}%
A function $U=(u,u|_{\partial \Omega })\in \mathcal{V}_{\rho }\cap\mathbb{X}%
_2$ is said to be a weak solution of \eqref{3.4}, if for every $%
V=(v,v|_{\partial \Omega })\in \mathcal{V}_{\rho }\cap \mathbb{X}_2$, there
holds%
\begin{equation}
\mathcal{A}_{\rho }(U,V):=\int_{\Omega }\left\vert \nabla u\right\vert
^{p-2}\nabla u\cdot \nabla v\;dx+\rho \int_{\partial \Omega }\left\vert
\nabla _{\Gamma }u\right\vert ^{q-2}\nabla _{\Gamma }u\cdot \nabla _{\Gamma
}v\;d\sigma =0.  \label{weak-1}
\end{equation}%
Let $C:=(c,c)$ with $c\in \mathbb{R}$. Then it is clear that $C\in \mathcal{N%
}\left( A_{\rho }\right) $.

Conversely, let $U=(u,u|_{\partial \Omega })\in \mathcal{N}\left( A_{\rho
}\right) $. Then, it follows from \eqref{weak-1} that%
\begin{equation}
\mathcal{A}_{\rho }(U,U):=\int_{\Omega }\left\vert \nabla u\right\vert
^{p}\;dx+\rho \int_{\partial \Omega }\left\vert \nabla _{\Gamma
}u\right\vert ^{q}\;d\sigma =0.  \label{weak2}
\end{equation}%
Since $\Omega $ is bounded and connected, this implies that $u$ is equal to
a constant. Therefore, $U=C\mathbf{1}$ and this completes the proof.
\end{proof}

\begin{lemma}
\label{lem4.5} The range of the operator $A_{\rho }$ is given by 
\begin{equation*}
\mathcal{R}(A_{\rho })=\left\{ F:=(f,g)\in \mathbb{X}_{2}:\;\int_{\overline{%
\Omega }}F\;d\mu :=\int_{\Omega }f\;dx+\int_{\partial \Omega }g\;\frac{%
d\sigma }{b(x)}=0\right\} .
\end{equation*}
\end{lemma}

\begin{proof}
Let $F\in \mathcal{R}(A_{\rho })\subset \mathbb{X}_{2}$. Then there exists $%
U=(u,u|_{\partial \Omega })\in D(A_{\rho })$ such that $A_{\rho }(U)=F$.
More precisely, for every $V=(v,v|_{\partial \Omega })\in \mathcal{V}_{\rho
}\cap\mathbb{X}_2,$ we have%
\begin{align}
\mathcal{A}_{\rho }(U,V)& =\int_{\Omega }\left\vert \nabla u\right\vert
^{p-2}\nabla u\cdot \nabla v\;dx+\rho \int_{\partial \Omega }\left\vert
\nabla _{\Gamma }u\right\vert ^{q-2}\nabla _{\Gamma }u\cdot \nabla _{\Gamma
}v\;d\sigma  \label{weak} \\
& =\int_{\overline{\Omega }}FV\;d\mu .  \notag
\end{align}%
Taking $V=(1,1)\in \mathcal{V}_{\rho }\cap\mathbb{X}_2$, we obtain that $%
\displaystyle\int_{\overline{\Omega }}F\;d\mu =0$. Hence, 
\begin{equation*}
\displaystyle\mathcal{R}(A_{\rho })\subseteq \left\{ F\in \mathbb{X}%
_{2}:\;\int_{\overline{\Omega }}F\;d\mu =0\right\} .
\end{equation*}

Let us now prove the converse. To this end, let $F\in \mathbb{X}_{2}$ be
such that $\displaystyle\int_{\overline{\Omega }}F\;d\mu =0$. We have to
show that $F\in \mathcal{R}(A_{\rho })$, that is, there exists $U\in 
\mathcal{V}_{\rho }\cap \mathbb{X}_{2}$ such that (\ref{weak}) holds, for
every $V\in \mathcal{V}_{\rho }\cap \mathbb{X}_{2}$. To this end, consider%
\begin{equation*}
\mathcal{V}_{\rho ,0}:=\left\{ U=(u,u|_{\partial \Omega })\in \mathcal{V}%
_{\rho }\cap \mathbb{X}_{2}:\;\int_{\overline{\Omega }}U\;d\mu
:=\int_{\Omega }u\;dx+\int_{\partial \Omega }u\frac{d\sigma }{b}=0\right\} .
\end{equation*}%
It is clear that $\mathcal{V}_{\rho ,0}$ is a closed linear subspace of $%
\mathcal{V}_{\rho }\cap \mathbb{X}_{2}\hookrightarrow \mathbb{X}_{2}$, and
therefore is a reflexive Banach space. Using \cite[Section 1.1]{Maz85}, we
have that the norm%
\begin{equation*}
\Vert U\Vert _{\mathcal{V}_{\rho ,0}}:=\Vert \nabla u\Vert _{p,\Omega }+\rho
\Vert \nabla _{\Gamma }u\Vert _{q,\partial \Omega }
\end{equation*}%
defines an equivalent norm on $\mathcal{V}_{\rho ,0}$. Hence, there exists a
constant $C>0$ such that for every $U\in \mathcal{V}_{\rho ,0}$, 
\begin{equation}
\Vert |U\Vert |_{2}\leq C\Vert U\Vert _{\mathcal{V}_{\rho ,0}}:=\Vert \nabla
u\Vert _{p,\Omega }+\rho \Vert \nabla _{\Gamma }u\Vert _{q,\partial \Omega }.
\label{norm}
\end{equation}%
Define the functional $\mathcal{F}_{\rho }:\;\mathcal{V}_{\rho
,0}\rightarrow \mathbb{R}$ by 
\begin{equation*}
\mathcal{F}_{\rho }(U)=\frac{1}{p}\int_{\Omega }|\nabla u|^{p}\;dx+\frac{%
\rho }{q}\int_{\partial \Omega }|\nabla _{\Gamma }u|^{q}\;d\sigma -\int_{%
\overline{\Omega }}FU\;d\mu .
\end{equation*}%
It is easy to see that $\mathcal{F}_{\rho }$ is convex and
lower-semicontinuous on $\mathbb{X}_{2}$ (see Proposition \ref{PCLS}). We
show now that $\mathcal{F}_{\rho }$ is coercive. By exploiting a classical H%
\"{o}lder inequality and using \eqref{norm}, we have%
\begin{align}
\left\vert \int_{\overline{\Omega }}FU\;d\mu \right\vert & \leq C\Vert
|F\Vert |_{2}\Vert |U\Vert |_{2}\leq C\Vert |F\Vert |_{2}\Vert U\Vert _{%
\mathcal{V}_{\rho ,0}}  \notag  \label{est11} \\
& =C\Vert |F\Vert |_{2}\left( \Vert \nabla u\Vert _{p,\Omega }+\rho \Vert
\nabla _{\Gamma }u\Vert _{q,\partial \Omega }\right) .  \notag
\end{align}%
Obviously, this estimate yields%
\begin{equation}
-\int_{\overline{\Omega }}FU\;d\mu \geq -C\Vert |F\Vert |_{2}\left( \Vert
\nabla u\Vert _{p,\Omega }+\rho \Vert \nabla _{\Gamma }u\Vert _{q,\partial
\Omega }\right) .  \label{est12}
\end{equation}%
Therefore, from (\ref{est12}), we immediately get%
\begin{equation*}
\frac{\mathcal{F}_{\rho }(U)}{\Vert U\Vert _{\mathcal{V}_{\rho ,0}}}\geq 
\frac{\frac{1}{p}\Vert \nabla u\Vert _{p,\Omega }^{p}+\frac{\rho }{q}\Vert
\nabla _{\Gamma }u\Vert _{q,\partial \Omega }^{q}}{{\Vert \nabla u\Vert
_{p,\Omega }+\rho \Vert \nabla _{\Gamma }u\Vert _{q,\partial \Omega }}}%
-C\Vert |F\Vert |_{2}.
\end{equation*}%
This inequality implies that 
\begin{equation*}
\lim_{\Vert U\Vert _{\mathcal{V}_{\rho ,0}}\rightarrow +\infty }\frac{%
\mathcal{F}_{\rho }(U)}{\Vert U\Vert _{\mathcal{V}_{\rho ,0}}}=+\infty ,
\end{equation*}%
and this shows that the functional $\mathcal{F}_{\rho }$ is coercive. Since $%
\mathcal{F}_{\rho }$ is also convex, lower-semicontinuous, it follows from 
\cite[Theorem 3.3.4]{Gui} that, there exists a function $U^{\ast }\in 
\mathcal{V}_{\rho ,0}$ which minimizes $\mathcal{F}_{\rho }$. More
precisely, for all $V\in \mathcal{V}_{\rho ,0}$, $\mathcal{F}_{\rho
}(U^{\ast })\leq \mathcal{F}_{\rho }(V);$ this implies that for every $%
0<t\leq 1$ and every $V\in \mathcal{V}_{\rho ,0}$,%
\begin{equation*}
\mathcal{F}_{\rho }(U^{\ast }+tV)-\mathcal{F}_{\rho }(U^{\ast })\geq 0.
\end{equation*}%
Hence, 
\begin{equation*}
\lim_{t\downarrow 0}\frac{\mathcal{F}_{\rho }(U^{\ast }+tV)-\mathcal{F}%
_{\rho }(U^{\ast })}{t}\geq 0.
\end{equation*}%
Using the Lebesgue Dominated Convergence, an easy computation shows that%
\begin{align}
0\leq \lim_{t\downarrow 0}\frac{\mathcal{F}_{\rho }(U^{\ast }+tV)-\mathcal{F}%
_{\rho }(U^{\ast })}{t}& =\int_{\Omega }|\nabla u^{\ast }|^{p-2}\nabla
u^{\ast }\cdot \nabla v\;dx  \label{ineq8} \\
& +\rho \int_{\partial \Omega }|\nabla _{\Gamma }u^{\ast }|^{q-2}\nabla
_{\Gamma }u^{\ast }\cdot \nabla _{\Gamma }v\;d\sigma -\int_{\overline{\Omega 
}}FV\;d\mu .  \notag
\end{align}%
Changing $V$ to $-V$ into (\ref{ineq8}) gives that%
\begin{equation}
\int_{\Omega }|\nabla u^{\ast }|^{p-2}\nabla u^{\ast }\cdot \nabla
v\;dx+\rho \int_{\partial \Omega }|\nabla _{\Gamma }u^{\ast }|^{q-2}\nabla
_{\Gamma }u^{\ast }\cdot \nabla _{\Gamma }v\;d\sigma =\int_{\overline{\Omega 
}}FV\;d\mu ,  \label{ineq9}
\end{equation}%
for every $V\in \mathcal{V}_{\rho ,0}$. Now, let $V\in \mathcal{V}_{\rho
}\cap \mathbb{X}_{2}$. Writing $V=V-C+C$ with $C=(c,c),$%
\begin{equation*}
c:=\frac{1}{\left( \lambda _{1}+\lambda _{2}\right) }\left( \int_{\Omega
}v\;dx+\int_{\partial \Omega }v\;\frac{d\sigma }{b}\right) ,
\end{equation*}%
and using the fact that $\displaystyle\int_{\overline{\Omega }}F\;d\mu =0$,
we obtain, for every $V\in \mathcal{V}_{\rho }\cap \mathbb{X}_{2}$, that%
\begin{equation*}
\int_{\Omega }|\nabla u^{\ast }|^{p-2}\nabla u^{\ast }\cdot \nabla
v\;dx+\rho \int_{\partial \Omega }|\nabla _{\Gamma }u^{\ast }|^{q-2}\nabla
_{\Gamma }u^{\ast }\cdot \nabla _{\Gamma }v\;d\sigma =\int_{\overline{\Omega 
}}FV\;d\mu .
\end{equation*}%
Therefore, $A_{\rho }(U)=F$. Hence, $F\in \mathcal{R}(A_{\rho })$ and this
completes the proof of the lemma.
\end{proof}

The following result is a direct consequence of Lemmas \ref{lem4.5bis}, \ref%
{lem4.5}. This is the main result of this section.

\begin{theorem}
\label{QFA} The operator $A_{\rho }=\partial \mathcal{J}_{\rho }$ satisfies
the following type of "quasi-linear" Fredholm alternative:%
\begin{equation*}
\mathcal{R}\left( A_{\rho }\right) =\mathcal{N}\left( A_{\rho }\right)
^{\perp }=\left\{ F\in \mathbb{X}_{2}:\left\langle F,\mathbf{1}\right\rangle
_{\mathbb{X}_{2}}=0\right\} .
\end{equation*}
\end{theorem}

\section{Necessary and sufficient conditions for existence of solutions}

\label{main}

In this section, we prove the first main result (cf. Theorem \ref{main*})\
for problem (\ref{1.1}). Before we do so, we will need the following results
from maximal monotone operators theory and convex analysis.

\begin{definition}
\label{D1} Let $\mathcal{H}$ be a real Hilbert space. Two subsets $K_{1}$
and $K_{2}$ of $\mathcal{H}$ are said to be almost equal, written, $%
K_{1}\simeq K_{2},$ if $K_{1}$ and $K_{2}$ have the same closure and the
same interior, that is, $\overline{K_{1}}=\overline{K_{2}}$ and $\mbox{int}%
\left( K_{1}\right) =\mbox{int}\left( K_{2}\right) .$
\end{definition}

The following abstract result is taken from \cite[Theorem 3 and
Generalization in p.173--174]{BH}.

\begin{theorem}[\textbf{Brezis-Haraux}]
\label{T2} Let $A$ and $B$ be subdifferentials of proper convex lower
semicontinuous functionals $\varphi_1$ and $\varphi_2$, respectively, on a
real Hilbert space $\mathcal{H}$ with $D(\varphi_1)\cap D(\varphi_2)\ne
\emptyset$, and let $C$ be the subdifferential of the proper, convex lower
semicontinuous functional $\varphi_1+\varphi_2$, that is $%
C=\partial(\varphi_1+\varphi_2)$. Then 
\begin{equation*}
\mathcal{R}(A)+\mathcal{R}(B)\subset \overline{\mathcal{R}(C)}\;\;\;%
\mbox{
and } \;\;\; \mbox{Int}\left(\mathcal{R}(A)+\mathcal{R}(B)\right)\subset%
\mathcal{R}(C)
\end{equation*}
In particular, if the operator $A+B$ is maximal monotone, then 
\begin{equation*}
\mathcal{R}\left( A+B\right) \simeq \mathcal{R}\left( A\right) +\mathcal{R}%
\left( B\right),
\end{equation*}
and this is the case if $\partial(\varphi_1+\varphi_2)=\partial\varphi_1+
\partial\varphi_2$.
\end{theorem}

\subsection{Assumptions and intermediate results}

Let us recall that the aim of this section is to establish some necessary
and sufficient conditions for the solvability of the following nonlinear
elliptic problem:%
\begin{equation}
\begin{cases}
-\Delta _{p}u+\alpha _{1}\left( u\right) =f, & \text{ in }\Omega , \\ 
&  \\ 
b\left( x\right) \left\vert \nabla u\right\vert ^{p-2}\partial _{n}u-\rho
b\left( x\right) \Delta _{q,\Gamma }u+\alpha _{2}\left( u\right) =g, & \text{
on }\partial \Omega ,%
\end{cases}
\label{3.1}
\end{equation}%
where $p,q\in (1,+\infty )$ are fixed.% and satisfy \eqref{p-q}.
We also assume that $\alpha _{j}:\mathbb{R}\rightarrow \mathbb{R}$ ($j=1,2$)
satisfy the following assumptions.

\begin{assumption}
\label{assump-1} The functions $\alpha _{j}:\;\mathbb{R}\rightarrow \mathbb{R%
}$ $(j=1,2)$ are odd, monotone nondecreasing, continuous and satisfy $\alpha
_{j}(0)=0$.
\end{assumption}

Let $\tilde{\alpha}_{j}$ be the inverse of $\alpha _{j}$. We define the
functions $\Lambda _{j},\;\widetilde{\Lambda }_{j}:\;\mathbb{R}\rightarrow 
\mathbb{R}_{+}$ ($j=1,2$) by%
\begin{equation}
\Lambda _{j}(t):=\int_{0}^{|t|}\alpha _{j}(s)ds\;\mbox{ and }\;\widetilde{%
\Lambda }_{j}(t):=\int_{0}^{|t|}\tilde{\alpha}_{j}(s)ds.  \label{Lambda-j}
\end{equation}%
Then it is clear that $\Lambda _{j}$, $\widetilde{\Lambda }_{j}$ are even,
convex and monotone increasing on $\mathbb{R}_{+},$\ with $\Lambda _{j}(0)=%
\widetilde{\Lambda }_{j}(0),$ for each $j=1,2$. Moreover, since $\alpha _{j}$
are odd, we have $\Lambda _{j}^{^{\prime }}\left( t\right) =\alpha
_{j}\left( t\right) ,$ for all $t\in \mathbb{R}$ and $j=1,2$, with a similar
relation holding for $\widetilde{\Lambda }_{j}$ as well. The following
result whose proof is included in \cite[Chap. I, Section 1.3, Theorem 3]%
{RR-1} holds.

\begin{lemma}
\label{lemma-4-4} The functions $\Lambda _{j}$ and $\widetilde{\Lambda }_{j}$
$(j=1,2)$ satisfy \eqref{st} and \eqref{tt}. More precisely, for all $s,t\in 
\mathbb{R}$,%
\begin{equation*}
st\leq \Lambda _{j}(s)+\widetilde{\Lambda }_{j}(t).
\end{equation*}%
If $s=\alpha _{j}(t)$ or $t=\tilde{\alpha}_{j}(s),$ then we also have
equality, that is,%
\begin{equation*}
\widetilde{\Lambda }_{j}(\alpha _{j}(s))=s\alpha _{j}(s)-\Lambda _{j}(s),%
\text{ }j=1,2.
\end{equation*}
\end{lemma}

We note that in \cite{RR-1}, the statement of Lemma \ref{lemma-4-4} assumed
that $\Lambda _{j}$, $\widetilde{\Lambda }_{j}$ are $\mathcal{N}$-functions
in the sense of Definition \ref{N-func}. However, the conclusion of that
result holds under the weaker hypotheses of Lemma \ref{lemma-4-4}. 
%\begin{proof}
%Indeed, it suffices to consider $t,$ $s\geq 0$. First, note that%
%\begin{equation}
%\tilde{\alpha}_{j}(s)=\inf \{t:\;\alpha _{j}(t)>s\},\;\;s\geq 0.
%\label{inf-def}
%\end{equation}%
%Define the sets%
%\begin{align*}
%Q_{\left( s,t\right) }^{1}& =\left\{ \left( u,v\right) :u\leq s,\text{ }%
%v\leq t,\text{ }0\leq u\leq \alpha _{j}(v),\;0\leq \tilde{\alpha}%
%_{j}(u)<v\right\} , \\
%Q_{\left( s,t\right) }^{2}& =\left\{ \left( u,v\right) :u\leq s,\text{ }%
%v\leq t\text{, }u>\alpha _{j}(v)\geq 0,\;0\leq v\leq \tilde{\alpha}%
%_{j}(u)\right\} .
%\end{align*}%
%Next, using (\ref{Lambda-j}) and employing the Fubini-Tonelly theorem, we
%obtain%
%\begin{align}
%0\leq st& =\int_{0}^{s}\int_{0}^{t}dvdu=\int \int_{Q_{\left( s,t\right)
%}^{1}}dvdu+\int \int_{Q_{\left( s,t\right) }^{2}}dvdu  \label{equality8} \\
%& =\int_{0}^{s}\;du\int_{0}^{t\wedge \alpha
%_{j}(u)}dv+\int_{0}^{t}dv\int_{0}^{s\wedge \tilde{\alpha}_{j}(v)}du  \notag
%\\
%& \leq \int_{0}^{s}\alpha _{j}(u)du+\int_{0}^{t}\tilde{\alpha}%
%_{j}(v)dv=\Lambda _{j}(s)+\widetilde{\Lambda }_{j}(t),  \notag
%\end{align}%
%where $u\wedge v=\min \left( u,v\right) $. Equality in (\ref{equality8})
%occurs if and only if $t\geq \alpha _{j}(u)$ so that $\tilde{\alpha}%
%_{j}(v)=s $ by \eqref{inf-def}, or $t=\alpha _{j}(s)$ and $s\geq \tilde{%
%\alpha}_{j}(t)$. This completes the proof of the lemma.
%\end{proof}

Define the functional $\mathcal{J}_{2}:\;\mathbb{X}_{2}\rightarrow \lbrack
0,+\infty ]$ by 
\begin{equation*}
\mathcal{J}_{2}(u,v):=%
\begin{cases}
\displaystyle\int_{\Omega }\Lambda _{1}(u)\;dx+\int_{\partial \Omega
}\Lambda _{2}(v)\;\frac{d\sigma }{b},\;\;\; & \mbox{ if }\;(u,v)\in D(%
\mathcal{J}_{2}), \\ 
+\infty , & \mbox{ if }(u,v)\in \mathbb{X}_{2}\setminus D(\mathcal{J}_{2}),%
\end{cases}%
\end{equation*}%
with the effective domain%
\begin{equation*}
D(\mathcal{J}_{2}):=\left\{ (u,v)\in \mathbb{X}_{2}:\;\int_{\Omega }\Lambda
_{1}(u)\;dx+\int_{\partial \Omega }\Lambda _{2}(v)\;\frac{d\sigma }{b}%
<\infty \right\} .
\end{equation*}

\begin{lemma}
Let $\alpha _{j}$ $(j=1,2)$ satisfy Assumption \ref{assump-1}. Then the
functional $\mathcal{J}_{2}$ is proper, convex and lower semicontinuous on $%
\mathbb{X}_{2}$.
\end{lemma}

\begin{proof}
It is routine to check that $\mathcal{J}_{2}$ is convex and proper. This
follows easily from the convexity of $\Lambda _{j}$ and the fact that $%
\Lambda _{j}(0)=0$. To show the lower semicontinuity on $\mathbb{X}_{2}$,
let $U_{n}=(u_{n},v_{n})\in D(\mathcal{J}_{2})$ be such that $%
U_{n}\rightarrow U:=(u,v)$ in $\mathbb{X}_{2}$ and $\mathcal{J}%
_{2}(U_{n})\leq C$ for some constant $C>0$. Since $U_{n}\rightarrow U$ in $%
\mathbb{X}_{2}$, then there is a subsequence, which we also denote by $%
U_{n}=(u_{n},v_{n}),$ such that $u_{n}\rightarrow u$ a.e. on $\Omega $ and $%
v_{n}\rightarrow v$ $\sigma $-a.e. on $\Gamma $. Since $\Lambda _{j}(\cdot )$
are continuous (thus, lower-semicontinuous), we have%
\begin{equation*}
\Lambda _{1}(u)\leq \liminf_{n\rightarrow \infty }\Lambda _{1}(u_{n})\;%
\mbox{ and }\;\Lambda _{2}(v)\leq \liminf_{n\rightarrow \infty }\Lambda
_{2}(v_{n}).
\end{equation*}%
By Fatou's Lemma, we obtain 
\begin{equation*}
\int_{\Omega }\Lambda _{1}(u)dx\leq \int_{\Omega }\liminf_{n\rightarrow
\infty }\Lambda _{1}(u_{n})dx\leq \liminf_{n\rightarrow \infty }\int_{\Omega
}\Lambda _{1}(u_{n})dx
\end{equation*}%
and 
\begin{equation*}
\int_{\partial \Omega }\Lambda _{2}(v)\frac{d\sigma }{b}\leq \int_{\partial
\Omega }\liminf_{n\rightarrow \infty }\Lambda _{2}(v_{n})\frac{d\sigma }{b}%
\leq \liminf_{n\rightarrow \infty }\int_{\partial \Omega }\Lambda _{2}(v_{n})%
\frac{d\sigma }{b}.
\end{equation*}%
Hence, $\mathcal{J}_{2}$ is lower semicontinuous on $\mathbb{X}_{2}$.
\end{proof}

We have the following result whose proof is contained in \cite[Chap. III,
Section 3.1, Theorem 2]{RR-1}.

\begin{lemma}
\label{vector-space} Let $\alpha _{j}$ $(j=1,2)$ satisfy Assumption \ref%
{assump-1} and assume that there exist constants $C_{j}>1$ $(j=1,2)$ such
that 
\begin{equation}
\Lambda _{j}(2t)\leq C_{j}\Lambda _{j}(t),\;\mbox{ for all }\;t\in \mathbb{R}%
.  \label{mod-delta}
\end{equation}%
Then $D(\mathcal{J}_{2})$ is a vector space.
\end{lemma}

Let the operator $B_{2}$ be defined by 
\begin{equation}
\begin{cases}
D\left( B_{2}\right) =\left\{ U:=\left( u,v\right) \in \mathbb{X}_{2}:\left(
\alpha _{1}\left( u\right) ,\alpha _{2}\left( v\right) \right) \in \mathbb{X}%
_{2}\right\} , \\ 
B_{2}(U)=\left( \alpha _{1}\left( u\right) ,\alpha _{2}\left( v\right)
\right) .%
\end{cases}
\label{3.3}
\end{equation}

We have the following result.

\begin{lemma}
\label{sub-1} Let the assumptions of Lemma \ref{vector-space} be satisfied.
Then the subdifferential of $\mathcal{J}_{2}$ and the operator $B_{2}$
coincide, that is, for all $(u,v)\in D(B_{2})=D(\partial \mathcal{J}_{2})$,%
\begin{equation*}
\partial \mathcal{J}_{2}(u,v)=B_{2}(u,v).
\end{equation*}
\end{lemma}

\begin{proof}
Let $U=(u,v)\in D(\mathcal{J}_{2})$ and $F=(f,g)\in \partial \mathcal{J}%
_{2}(u,v)$. Then by definition, $F\in \mathbb{X}_{2}$ and, for every $%
V=(u_{1},v_{1})\in D(\mathcal{J}_{2}),$ we get%
\begin{equation*}
\int_{\overline{\Omega }}F(V-U)\;d\mu \leq \mathcal{J}_{2}(V)-\mathcal{J}%
_{2}(U).
\end{equation*}%
Let $V=U+tW,$ with $W=(u_{2},v_{2})\in D(\mathcal{J}_{2})$ and $0<t\leq 1$.
Then by Lemma \ref{vector-space}, $V=U+tW\in D(\mathcal{J}_{2})$. Now,
dividing by $t$ and taking the limit as $t\downarrow 0$, we obtain%
\begin{equation}
\int_{\Omega }FW\;d\mu \leq \int_{\Omega }\alpha
_{1}(u)u_{2}dx+\int_{\partial \Omega }\alpha _{2}(v)v_{2}\,\frac{d\sigma }{b}%
.  \label{first}
\end{equation}%
Changing $W$ to $-W$ in (\ref{first}) gives that%
\begin{equation*}
\int_{\overline{\Omega }}FW\;d\mu =\int_{\Omega }\alpha
_{1}(u)u_{2}dx+\int_{\partial \Omega }\alpha _{2}(v)v_{2}\,\frac{d\sigma }{b}%
.
\end{equation*}%
In particular, if $W=(u_{2},0)$ with $u_{2}\in \mathcal{D}(\Omega )$, we have%
\begin{equation*}
\int_{\Omega }fu_{2}\;dx=\int_{\Omega }\alpha _{1}(u)u_{2}\;dx,
\end{equation*}%
and this shows that $\alpha _{1}(u)=f$. Similarly, one obtains that $\alpha
_{2}(v)=g$. We have shown that $U\in D(B_{2})$ and%
\begin{equation*}
B_{2}(U):=B_{2}(u,v)=(\alpha _{1}(u),\alpha _{2}(v))=(f,g).
\end{equation*}%
Conversely, let $U=(u,v)\in D(B_{2})$ and set $F=(f,g):=B_{2}(u,v)=(\alpha
_{1}(u),\alpha _{2}(v))$. Since $(\alpha _{1}(u),\alpha _{2}(v))\in \mathbb{X%
}_{2}$, from (\ref{Lambda-j}) and (\ref{mod-delta}), it follows that%
\begin{equation*}
\int_{\Omega }\Lambda _{1}(u)dx+\int_{\partial \Omega }\Lambda _{2}(v)\frac{%
d\sigma }{b}<\infty .
\end{equation*}%
Hence, $U=(u,v)\in D(\mathcal{J}_{2})$. Let $V=(u_{1},v_{1})\in D(\mathcal{J}%
_{2})$. Using Lemma \ref{lemma-4-4}, we obtain%
\begin{align}
\alpha _{1}(u)(u_{1}-u)& =\alpha _{1}(u)u_{1}-\alpha _{1}(u)u \\
& \leq \Lambda _{1}(u_{1})+\Lambda _{1}(\alpha _{1}(u))-\alpha _{1}(u)u 
\notag \\
& =\Lambda _{1}(u_{1})-\Lambda _{1}(u)  \notag
\end{align}%
and similarly,%
\begin{equation*}
\alpha _{2}(v)(v_{1}-v)\leq \Lambda _{2}(v_{1})-\Lambda _{2}(v).
\end{equation*}%
Therefore,%
\begin{align*}
\int_{\overline{\Omega }}F(V-U)\;d\mu & =\int_{\Omega }\alpha
_{1}(u)(u_{1}-u)dx+\int_{\partial \Omega }\alpha _{2}(v)(v_{1}-v)\frac{%
d\sigma }{b} \\
& \leq \mathcal{J}_{2}(V)-\mathcal{J}_{2}(U).
\end{align*}%
By definition, this shows that $F=(\alpha _{1}(u),\alpha
_{2}(v))=B_{2}(U)\in \partial \mathcal{J}_{2}(U)$. We have shown that $U\in
D(\partial \mathcal{J}_{2})$ and $B_{2}(U)\in \partial \mathcal{J}_{2}(U)$.
This completes the proof of the lemma.
\end{proof}

Next, we define the functional $\mathcal{J}_{3,\rho }:\;\mathbb{X}%
_{2}\rightarrow \lbrack 0,+\infty ]$ by%
\begin{equation}
\mathcal{J}_{3,\rho }(U)=%
\begin{cases}
\mathcal{J}_{\rho }(U)+\mathcal{J}_{2}(U)\;\; & \mbox{ if }\;U\in D(\mathcal{%
J}_{3,\rho }):=D(\mathcal{J}_{\rho })\cap D(\mathcal{J}_{2}), \\ 
+\infty & \mbox{ if }\;U\in \mathbb{X}_{2}\backslash D(\mathcal{J}_{3,\rho
}).%
\end{cases}%
\end{equation}%
Note that for $\rho =0$,%
\begin{equation}
D(\mathcal{J}_{3,0})=\{U=(u,u|_{\partial \Omega })\in D(\mathcal{J}%
_{2}):u\in W^{1,p}(\Omega )\cap L^2(\Omega),\;u|_{\partial\Omega}\in
L^2(\partial\Omega)\},  \label{dom1}
\end{equation}%
while for $\rho =1$, 
\begin{equation}
D(\mathcal{J}_{3,1})=\{U=(u,u|_{\partial \Omega })\in D(\mathcal{J}%
_{2}):u\in W^{1,p}(\Omega )\cap L^2(\Omega),\;u|_{\partial \Omega }\in
W^{1,q}(\partial \Omega )\cap L^2(\partial\Omega)\}.  \label{dom2}
\end{equation}%
We have the following result.

\begin{lemma}
\label{sub-sum} Let the assumptions of Lemma \ref{vector-space} be
satisfied. Then the subdifferential of the functional $\mathcal{J}_{3,\rho }$
is given by%
\begin{align*}
D(\partial \mathcal{J}_{3,\rho })& =\left\{ U=(u,u_{\mid \partial \Omega
})\in D(\mathcal{J}_{3,\rho }):-\Delta _{p}u+\alpha _{1}(u)\in L^{2}(\Omega
)\right. \\
& \left. \text{and }b(x)|\nabla u|^{p-2}\partial _{n}u-b(x)\rho \Delta
_{q,\Gamma }u+\alpha _{2}(u)\in L^{2}(\partial \Omega ,{d\sigma }/{b}%
)\right\}
\end{align*}%
and 
\begin{equation}
\partial \mathcal{J}_{3,\rho }(U)=\bigg(-\Delta _{p}u+\alpha
_{1}(u),b(x)|\nabla u|^{p-2}\partial _{n}u-b(x)\rho \Delta _{q,\Gamma
}u+\alpha _{2}(u)\bigg).  \label{dom-sub-J3-2}
\end{equation}%
In particular, if for every $U=(u,u_{\mid \partial \Omega })\in D(\mathcal{J}%
_{3,\rho }),$ the function $\left( \alpha _{1}(u),\alpha _{2}(u)\right) \in 
\mathbb{X}_{2}$, then 
\begin{equation*}
\partial \mathcal{J}_{3,\rho }:=\partial (\mathcal{J}_{\rho }+\mathcal{J}%
_{2})=\partial \mathcal{J}_{\rho }+\partial \mathcal{J}_{2}.
\end{equation*}
\end{lemma}

\begin{proof}
We calculate the subdifferential $\partial \mathcal{J}_{3,\rho }$. Let $%
F=(f,g)\in \partial \mathcal{J}_{3,\rho }(U)$, that is, $F\in \mathbb{X}_{2}$%
, $U\in D(\mathcal{J}_{3,\rho })=D(\mathcal{J}_{\rho })\cap D(\mathcal{J}%
_{2})$ and for every $V\in D(\mathcal{J}_{3,\rho })$, we have%
\begin{equation*}
\int_{\overline{\Omega }}F(V-U)d\mu \leq \mathcal{J}_{3,\rho }(V)-\mathcal{J}%
_{3,\rho }(U).
\end{equation*}%
Proceeding as in Remark \ref{cal-sub} and the proof of Lemma \ref{sub-1}, we
obtain that%
\begin{equation*}
-\Delta _{p}u+\alpha _{1}(u)=f\;\;\mbox{ in }\;\mathcal{D}(\Omega )^{\prime
},
\end{equation*}%
and 
\begin{equation*}
b(x)|\nabla u|^{p-2}\partial _{n}u-b(x)\rho \Delta _{q,\Gamma }u+\alpha
_{2}(u)=g\;\mbox{ weakly on }\;\partial \Omega .
\end{equation*}%
Noting that $\partial \mathcal{J}_{3,\rho }$ is also a single-valued
operator (which follows from the assumptions on $\alpha _{j}$ and $\Lambda
_{j}$), we easily obtain \eqref{dom-sub-J3-2}, and this completes the proof
of the first part.

To show the last part, note that it is clear that $\partial \mathcal{J}%
_{\rho }+\partial \mathcal{J}_{2}\subset \partial \mathcal{J}_{3,\rho }$
always holds. To show the converse inclusion, let assume that for every $%
U=(u,u_{\mid \partial \Omega })\in D(\mathcal{J}_{3,\rho }),$ the function $%
(\alpha _{1}(u),\alpha _{2}(u))\in \mathbb{X}_{2}$. Then it follows from %
\eqref{J-rho}, \eqref{3.3} (since $\partial \mathcal{J}_{2}=B_{2}$) and %
\eqref{dom-sub-J3-2}, that $D(\partial \mathcal{J}_{3,\rho })=D(\partial 
\mathcal{J}_{\rho })\cap D(\partial \mathcal{J}_{2})$ and 
\begin{align*}
\partial \mathcal{J}_{3,\rho }(U)& =\left( -\Delta _{p}u+\alpha
_{1}(u),b(x)|\nabla u|^{p-2}\partial _{n}u-b(x)\rho \Delta _{q,\Gamma
}u+\alpha _{2}(u)\right)  \\
& =\left( -\Delta _{p}u,b(x)|\nabla u|^{p-2}\partial _{n}u-b(x)\rho \Delta
_{q,\Gamma }u\right) +\left( \alpha _{1}(u),\alpha _{2}(u)\right)  \\
& =\partial \mathcal{J}_{\rho }(U)+\partial \mathcal{J}_{2}(U).
\end{align*}%
This completes the proof.
\end{proof}

The following lemma is the main ingredient in the proof of Theorem \ref{T3}
below.

\begin{lemma}
Let $B_{1}:=A_{\rho }$ and set $B_{3}:=\partial \mathcal{J}_{3,\rho }$. Then 
\begin{equation}
\mathcal{R}\left( B_{1}\right) +\mathcal{R}\left( B_{2}\right) \subset 
\overline{\mathcal{R}(B_{3})}\;\mbox{ and }\;\mbox{Int}(\mathcal{R}\left(
B_{1}\right) +\mathcal{R}\left( B_{2}\right) )\subset \mathcal{R}(B_{3}).
\label{sum-subdi}
\end{equation}%
In particular, if for every $U=(u,u_{\mid \partial \Omega })\in D(\mathcal{J}%
_{3,\rho }),$ the function $(\alpha _{1}(u),\alpha _{2}(u))\in \mathbb{X}%
_{2} $, then 
\begin{equation}
\mathcal{R}\left( B_{3}\right) :=\mathcal{R}\left( B_{1}+B_{2}\right) \simeq 
\mathcal{R}\left( B_{1}\right) +\mathcal{R}\left( B_{2}\right) .  \label{3.9}
\end{equation}
\end{lemma}

\begin{proof}
By Remark \ref{cal-sub} and Lemmas \ref{sub-1}, \ref{sub-sum}, the operators 
$B_{1}$, $B_{2}$ and $B_{3}$ are subdifferentials of proper, convex and
lower semicontinuous functionals $\mathcal{J}_{\rho },\mathcal{J}_{2}$ and $%
\mathcal{J}_{\rho }+\mathcal{J}_{2}$, respectively, on $\mathbb{X}_{2}$.
Hence, $B_{1}$, $B_{2}$ and $B_{3}$ are maximal monotone operators. In
particular, if $(\alpha _{1}(u),\alpha _{2}(u))\in \mathbb{X}_{2}$, for
every $U=(u,u_{\mid \partial \Omega })\in D(\mathcal{J}_{3,\rho }),$ then by
Lemma \ref{sub-sum}, one has $B_{3}=B_{1}+B_{2}$. Now, the lemma follows
from the celebrated Brezis-Haraux result in Theorem \ref{T2}.
\end{proof}

\subsection{Statement and proof of the main result}

Next, let $\mathcal{V}_{\rho }:=D(\mathcal{J}_{3,\rho })$ be given by %
\eqref{dom1} if $\rho =0$ and by \eqref{dom2} if $\rho =1$.

\begin{definition}
\label{def-weak-sol} Let $F=(f,g)\in \mathbb{X}_{2}$. A function $u\in
W^{1,p}(\Omega )$ is said to be a weak solution of \eqref{3.1}, if $\alpha
_{1}(u)\in L^{1}(\Omega ),$ $\alpha _{2}(u)\in L^{1}(\partial \Omega )$, $%
u|_{\partial \Omega }\in W^{1,q}(\partial \Omega ),$ if $\rho >0$ and 
\begin{align}
& \int_{\Omega }|\nabla u|^{p-2}\nabla u\cdot \nabla vdx+\rho \int_{\partial
\Omega }|\nabla _{\Gamma }u|^{q-2}\nabla _{\Gamma }u\cdot \nabla _{\Gamma
}vd\sigma  \label{weak-sol} \\
& +\int_{\Omega }\alpha _{1}(u)vdx+\int_{\partial \Omega }\alpha _{2}(u)v%
\frac{d\sigma }{b}=\int_{\Omega }fvdx+\int_{\partial \Omega }gv\frac{d\sigma 
}{b},  \notag
\end{align}%
for every $v\in W^{1,p}(\Omega )\cap C(\overline{\Omega })$ with $%
v|_{\partial \Omega }\in W^{1,q}(\partial \Omega ),$ if $\rho >0$.
\end{definition}

Recall that $\lambda _{1}:=\int_{\Omega }dx$ and $\displaystyle\lambda
_{2}:=\int_{\partial \Omega }\frac{d\sigma }{b}$. We also define the average 
$\left\langle F\right\rangle _{\overline{\Omega }}$ of $F=\left( f,g\right) $
with respect to the measure $\mu ,$ as follows:%
\begin{equation*}
\left\langle F\right\rangle _{\overline{\Omega }}:=\frac{1}{\mu \left( 
\overline{\Omega }\right) }\int_{\overline{\Omega }}Fd\mu =\frac{1}{\mu
\left( \overline{\Omega }\right) }\left( \int_{\Omega }fdx+\int_{\partial
\Omega }g\frac{d\sigma }{b}\right) ,
\end{equation*}%
where $\mu \left( \overline{\Omega }\right) =\lambda _{1}+\lambda _{2}$.
Now, we are ready to state the main result of this section.

\begin{theorem}
\label{T3} Let $\alpha _{j}$ $(j=1,2)$ satisfy Assumption \ref{assump-1} and
assume that the functions $\Lambda _{j}$ $(j=1,2)$ satisfy \eqref{mod-delta}%
. Let $F=\left( f,g\right) \in \mathbb{X}_{2}$. The following hold:

\begin{enumerate}
\item Suppose that the nonlinear elliptic problem \eqref{3.1} possesses a
weak solution. Then%
\begin{equation}
\left\langle F\right\rangle _{\overline{\Omega }}\in \frac{\lambda _{1}%
\mathcal{R}\left( \alpha _{1}\right) +\lambda _{2}\mathcal{R}\left( \alpha
_{2}\right) }{\lambda _{1}+\lambda _{2}}.  \label{3.7}
\end{equation}

\item Assume that 
\begin{equation}
\left\langle F\right\rangle _{\overline{\Omega }}\in \mbox{int}\left( \frac{%
\lambda _{1}\mathcal{R}\left( \alpha _{1}\right) +\lambda _{2}\mathcal{R}%
\left( \alpha _{2}\right) }{\lambda _{1}+\lambda _{2}}\right) .  \label{3.8}
\end{equation}%
Then the nonlinear elliptic problem \eqref{3.1} has at least one weak
solution.
\end{enumerate}
\end{theorem}

\begin{proof}
We show that condition \eqref{3.7} is necessary. Let $F:=(f,g)\in \mathbb{X}%
_{2}$ and let $U=\left(u,u_{\mid \partial \Omega }\right) \in
D(B_{3})\subset \mathcal{V}_{\rho }$ be a weak solution of $B_{3}U=F$. Then,
by definition, for every $V=(v,v|_{\partial \Omega })\in \mathcal{V}_{\rho
}, $ (\ref{weak-sol}) holds. Taking $v\equiv 1$ in (\ref{weak-sol}) yields%
\begin{equation*}
\int_{\Omega }f\;dx+\int_{\partial \Omega }g\;\frac{d\sigma }{b}%
=\int_{\Omega }\alpha _{1}\left( u\right) dx+\int_{\partial \Omega }\alpha
_{2}\left( u\right) \frac{d\sigma }{b}.
\end{equation*}%
Hence,%
\begin{equation*}
\int_{\Omega }f\;dx+\int_{\partial \Omega }g\;\frac{d\sigma }{b}\in
\left(\lambda _{1}\mathcal{R}\left( \alpha _{1}\right) +\lambda _{2}\mathcal{%
R}\left( \alpha _{2}\right)\right ),
\end{equation*}%
and so \eqref{3.7} holds. This completes the proof of part (a).\newline

We show that the condition \eqref{3.8} is sufficient.

(i) First, let $C\in \mathbf{C}$, where%
\begin{equation*}
\mathbf{C:}=\left\{ C=(c_{1},c_{2}):(c_{1},c_{2})\in \mathcal{R}(\alpha
_{1})\times \mathcal{R}(\alpha _{2})\right\} .
\end{equation*}%
By definition, one has that $\mathbf{C}\subset \mathcal{R}\left(
B_{2}\right) $ since $c_{1}=\alpha _{1}\left( d_{1}\right) $ for some
constant function $d_{1}$ on $\Omega $ and $c_{2}=\alpha _{2}(d_{2})$ for
some constant function $d_{2}$ on $\partial \Omega $. Let $F\in \mathbb{X}%
_{2}$ be such that (\ref{3.8}) holds. We must show $F\in \mathcal{R}\left(
B_{3}\right) $. By (\ref{3.8}), we may choose $C=\left( c_{1},c_{2}\right)
\in \mathbf{C}$ such that%
\begin{equation*}
\left\langle F\right\rangle _{\overline{\Omega }}=\frac{\lambda
_{1}c_{1}+\lambda _{2}c_{2}}{\lambda _{1}+\lambda _{2}}\in \mbox{int}\left( 
\frac{\lambda _{1}\mathcal{R}\left( \alpha _{1}\right) +\lambda _{2}\mathcal{%
R}\left( \alpha _{2}\right) }{\lambda _{1}+\lambda _{2}}\right) .
\end{equation*}%
Then, for $F\in \mathbb{X}_{2}$, we have $F=F_{1}+F_{2}$ with%
\begin{equation*}
F_{1}:=F-C\mbox{ and }\;F_{2}=C.
\end{equation*}%
First, $F_{1}\in \mathcal{R}\left( B_{1}\right) =\mathcal{N}\left(
B_{1}\right) ^{\perp }=\mathbf{1}^{\perp }$, since%
\begin{eqnarray*}
\displaystyle\int_{\overline{\Omega }}F_{1}d\mu &=&\int_{\overline{\Omega }%
}\left( F-C\right) d\mu \\
&=&\int_{\Omega }f\text{ }dx+\int_{\partial \Omega }g\text{ }\frac{d\sigma }{%
b}-\left( \lambda _{1}c_{1}+\lambda _{2}c_{2}\right) \\
&=&\left( \lambda _{1}+\lambda _{2}\right) \left\langle F\right\rangle _{%
\overline{\Omega }}-\left( \lambda _{1}c_{1}+\lambda _{2}c_{2}\right) =0.
\end{eqnarray*}%
Obviously, $F_{2}=C\in \mathcal{R}\left( B_{2}\right) $. Hence, it is
readily seen that 
\begin{equation*}
F\in (\mathcal{R}\left( B_{1}\right) +\mathcal{R}\left( B_{2}\right) ).
\end{equation*}

(ii) Next, denote by $\mathbb{B}_{\mathbb{R}}(x,r)$ the open ball in $%
\mathbb{R}$ of center $x$ and radius $r>0$. Since%
\begin{equation*}
\displaystyle\left\langle F\right\rangle _{\overline{\Omega }}\in \mbox{int}%
\left( \frac{\lambda _{1}\mathcal{R}\left( \alpha _{1}\right) +\lambda _{2}%
\mathcal{R}\left( \alpha _{2}\right) }{\lambda _{1}+\lambda _{2}}\right) ,
\end{equation*}%
there exists $\delta >0$ such that the open ball%
\begin{equation*}
\displaystyle\mathbb{B}_{\mathbb{R}}(\left\langle F\right\rangle _{\overline{%
\Omega }},\delta )\subset \left( \frac{\lambda _{1}\mathcal{R}\left( \alpha
_{1}\right) +\lambda _{2}\mathcal{R}\left( \alpha _{2}\right) }{\lambda
_{1}+\lambda _{2}}\right) .
\end{equation*}%
Since the mapping $F\mapsto \left\langle F\right\rangle _{\overline{\Omega }%
} $ from $\mathbb{X}_{2}$ into $\mathbb{R}$ is continuous, then there exists 
$\varepsilon >0$ such that 
\begin{equation*}
\left\langle G\right\rangle _{\overline{\Omega }}\in \mathbb{B}_{\mathbb{R}%
}(\left\langle F\right\rangle _{\overline{\Omega }},\delta )\subset \left( 
\frac{\lambda _{1}\mathcal{R}\left( \alpha _{1}\right) +\lambda _{2}\mathcal{%
R}\left( \alpha _{2}\right) }{\lambda _{1}+\lambda _{2}}\right) ,
\end{equation*}%
for all $G\in \mathbb{X}_{2}$ satisfying $\Vert |F-G\Vert |_{2}<\varepsilon $%
. It finally follows from part (i) above that $(\mathcal{R}\left(
B_{1}\right) +\mathcal{R}\left( B_{2}\right) )$ contains an $\varepsilon $%
-ball in $\mathbb{X}_{2}$ centered at $F$. Therefore, 
\begin{equation*}
F\in \mbox{int}(\mathcal{R}\left( B_{1}\right) +\mathcal{R}\left(
B_{2}\right) )\subset \mathcal{R}(B_{3}).
\end{equation*}%
%
%
%
%\subset\mbox{int}\left(\mathcal{R}(B_{3})\right) where the first inclusion follows from \eqref{sum-subdi}.
Consequently, problem (\ref{3.1}) is (weakly) solvable for every function $%
F=(f,g)\in \mathbb{X}_{2},$ if \eqref{3.8} holds. This completes the proof
of the theorem.
\end{proof}

\begin{remark}
It is important to remark that in order to prove Theorem \ref{T3}, we do not
require that $(\alpha _{1}(u),\alpha _{2}(u))$ should belong to $\mathbb{X}%
_{2},$ for every $U=(u,u_{\mid \Gamma })\in D(\mathcal{J}_{3,\rho }).$ In
particular, only the assumption (\ref{sum-subdi}) was needed. However, if
this happens, then we get the much stronger result in (\ref{3.9}) which
would require that the nonlinearities $\alpha _{1},\alpha _{2}$ satisfy
growth assumptions at infinity.
\end{remark}

We conclude this section with the following corollary and some examples.

\begin{corollary}
\label{Cor}Let the assumptions of Theorem \ref{T3} be satisfied. Let $%
F=\left( f,g\right) \in \mathbb{X}_{2}$. Assume that at least one of the
sets $\mathcal{R}\left( \alpha _{1}\right) $, $\mathcal{R}\left( \alpha
_{2}\right) $ is open. Then the nonlinear elliptic problem \eqref{3.1}
possesses a weak solution if and only if \eqref{3.8} holds.
\end{corollary}

\begin{remark}
Similar results to Theorem \ref{T3} and Corollary \ref{Cor} were also
obtained in \cite[Theorem 4.4]{GGGR}, but only when $p=q=2.$
\end{remark}

\subsection{Examples}

We will now give some examples as applications of Theorem \ref{T3}. Let$%
p,q\in (1,+\infty)$ be fixed. %and $q$ satisfy \eqref{p-q}.

\begin{example}
\label{e4.5}Let $\alpha _{1}\left( s\right) $ or $\alpha _{2}\left( s\right) 
$ be equal to $\alpha \left( s\right) =c\left\vert s\right\vert ^{r-1}s,$
where $c,$ $r>0$. Note that $\mathcal{R}\left( \alpha \right) =\mathbb{R}$.
It is easy to check that $\alpha $ satisfies all the conditions of
Assumption \ref{assump-1} and that the function $\Lambda
(t)=\int_{0}^{|t|}\alpha (s)ds$ satisfies (\ref{mod-delta}). Then, it
follows that problem (\ref{3.1}) is solvable for any $f\in L^{2}\left(
\Omega \right) ,$ $g\in L^{2}\left( \partial \Omega \right) $.
\end{example}

\begin{example}
\label{e4.6}Consider the case when $\rho =\alpha _{2}\equiv 0$ in (\ref{3.1}%
), that is, consider the following boundary value problem:%
\begin{equation*}
\left\{ 
\begin{array}{c}
-\Delta _{p}u+\alpha _{1}\left( u\right) =f\text{ in }\Omega , \\ 
b\left( x\right) \left\vert \nabla u\right\vert ^{p-2}\partial _{n}u=g \text{
on }\Gamma .%
\end{array}%
\right.
\end{equation*}%
Then, by Theorem \ref{T3}, this problem has a weak solution if%
\begin{equation*}
\int_{\Omega }f\text{ }dx+\int_{\partial \Omega }g\text{ }\frac{d\sigma }{b}%
\in \lambda _{1}\mbox{int}(\mathcal{R}\left( \alpha _{1}\right)) ,
\end{equation*}%
which yields the classical Landesman-Lazer result (see (\ref{1.6})) for $%
g\equiv 0$ and $p=2$.
\end{example}

\begin{example}
\label{e4.7}Let us now consider the case when $\alpha _{1}\equiv \alpha $
and $\alpha _{2}\equiv 0,$ where $\alpha $ is a continuous, odd and
nondecreasing function on $\mathbb{R}$ such that $\alpha \left( 0\right) =0$%
. The problem%
\begin{equation}
\left\{ 
\begin{array}{cc}
-\Delta _{p}u+\alpha \left( u\right) =f, & \text{in }\Omega , \\ 
b\left( x\right) \left\vert \nabla u\right\vert ^{p-2}\partial _{n}u-\rho
b\left( x\right) \Delta _{q,\Gamma }u=g, & \text{on }\partial \Omega ,%
\end{array}%
\right.  \label{e}
\end{equation}%
has a weak solution if%
\begin{equation}
\int_{\Omega }f\text{ }dx+\int_{\partial \Omega }g\text{ }\frac{d\sigma }{b}%
\in \lambda _{2}\mbox{int}\bigg(\mathcal{R}\left( \alpha \right)\bigg) .
\label{ee}
\end{equation}%
Let us now choose $\alpha \left( s\right) =\arctan \left( s\right) $ in (\ref%
{e}). Then, it is easy to check that%
\begin{equation*}
\Lambda (t):=\int_{0}^{|t|}\alpha (s)ds=\left\vert t\right\vert \arctan
\left( \left\vert t\right\vert \right) -\frac{1}{2}\ln \left( 1+t^{2}\right)
,\text{ }t\in \mathbb{R}
\end{equation*}%
is monotone increasing on $\mathbb{R}_{+}$ and that it satisfies $\Lambda
(2t)\leq C_{2}\Lambda (t),$\ $\forall t\in \mathbb{R}$, for some constant $%
C_{2}>1$. Therefore, (\ref{ee}) becomes the necessary and sufficient
condition%
\begin{equation}
\left\vert \frac{1}{\lambda _{2}}\left( \int_{\Omega }f\text{ }%
dx+\int_{\partial \Omega }g\text{ }\frac{d\sigma }{b}\right) \right\vert <%
\frac{\pi }{2}.
\end{equation}
\end{example}

\section{A priori estimates}

\label{priori}

Let $\Omega \subset \mathbb{R}^{N}$ be a bounded Lipschitz domain with
boundary $\partial \Omega $. Recall that $1<p,q<\infty $, $\rho \in \{0,1\}$
and $b\in L^{\infty }(\partial \Omega )$ with $b(x)\geq b_{0}>0,$ for some
constant $b_{0}$. We consider the nonlinear elliptic boundary value problem
formally given by 
\begin{equation}
\begin{cases}
\displaystyle-\Delta _{p}u+\alpha _{1}(x,u)+|u|^{p-2}u=f, & \mbox{ in }%
\;\Omega \\ 
&  \\ 
\displaystyle-\rho b(x)\Delta _{q,\Gamma }u+\rho b(x)|u|^{q-2}u+b(x)|\nabla
u|^{p-2}\partial _{n}u+\alpha _{2}(x,u)=g, & \mbox{ on }\;\partial \Omega ,%
\end{cases}
\label{eq-weak}
\end{equation}%
where $f\in L^{p_{1}}(\Omega )$ and $g\in L^{q_{1}}(\partial \Omega )$ for
some $1\leq p_{1},q_{1}\leq \infty $. If $\rho =0$, then the boundary
conditions in \eqref{eq-weak} are of Robin type. Existence and regularity of
weak solutions for this case have been obtained in \cite{BW} for $p=2$ (see
also \cite{W2} for the linear case) and for general $p$ in \cite{BW2}.
Therefore, we will concentrate our attention to the case $\rho =1$ only; in
this case, the boundary condition in (\ref{eq-weak}) is a generalized
Wentzell-Robin boundary condition. For the sake of simplicity, from now on
we will also take $b\equiv 1$.

\subsection{General assumptions}

Throughout this section, we assume that the functions $\alpha _{1}:\Omega
\times \mathbb{R}\rightarrow \mathbb{R}$ and $\alpha _{2}:\partial \Omega
\times \mathbb{R}\rightarrow \mathbb{R}$ satisfy the following conditions:

\begin{assumption}
\label{asump-51}%
\begin{equation*}
\begin{cases}
\displaystyle\alpha _{j}(x,\cdot )\mbox{ is odd and strictly increasing}, \\ 
\displaystyle\alpha _{j}(x,0)=0,\;\;\displaystyle\alpha _{j}(x,\cdot )%
\mbox{ is
continuous }, \\ 
\displaystyle\lim_{t\rightarrow 0}\frac{\alpha _{j}(x,t)}{t}=0,\;\;%
\displaystyle\lim_{t\rightarrow \infty }\frac{\alpha _{j}(x,t)}{t}=\infty,%
\end{cases}%
\end{equation*}%
for $\lambda _{N}$-a.e. $x\in \Omega $ if $j=1$ and $\sigma $-a.e. $x\in
\partial \Omega $ if $j=2$.
\end{assumption}

Since $\alpha _{j}(x,\cdot )$ are strictly increasing for $\lambda _{N}$%
-a.e. $x\in \Omega $ if $j=1$ and $\sigma $-a.e. $x\in \partial \Omega $ if $%
j=2$, then they have inverses which we denote by $\widetilde{\alpha }%
_{j}(x,\cdot )$ (cf. also Section 4). We define the functions $\Lambda _{1},$
$\widetilde{\Lambda }_{1}:\Omega \times \mathbb{R}\rightarrow \lbrack
0,\infty )$ and $\Lambda _{2},$ $\widetilde{\Lambda }_{2}:\partial \Omega
\times \mathbb{R}\rightarrow \lbrack 0,\infty )$ by 
\begin{equation*}
\Lambda _{j}(x,t):=\int_{0}^{|t|}\alpha _{j}(x,s)\;ds\;\mbox{ and
}\;\widetilde{\Lambda }_{j}(x,t):=\int_{0}^{|t|}\widetilde{\alpha }%
_{j}(x,s)\;ds.
\end{equation*}%
Then, it is clear that, for $\lambda _{N}$-a.e. $x\in \Omega $ if $j=1$ and $%
\sigma $-a.e. $x\in \partial \Omega $ if $j=2$, $\Lambda _{j}(x,\cdot )$ and 
$\widetilde{\Lambda }_{j}(x,\cdot )$ are differentiable, monotone and convex
with $\Lambda _{j}(x,0)=\widetilde{\Lambda }_{j}(x,0)=0.$ Furthermore, $%
\Lambda _{j}(x,\cdot )$ is an ${\mathcal{N}}$-function and $\widetilde{%
\Lambda }_{j}(x,\cdot )$ is its complementary ${\mathcal{N}}$-function. The
function $\widetilde{\Lambda }_{j}$ is then the complementary
Musielak-Orlick function of $\Lambda _{j}$ in the sense of Young (see
Definition \ref{def-24}).

\begin{assumption}
\label{asump-52} We assume, for $\lambda _{N}$-a.e. $x\in \Omega $ if $j=1$
and $\sigma $-a.e. $x\in \partial \Omega $ if $j=2$, that $\Lambda
_{j}(x,\cdot )$ and $\widetilde{\Lambda }_{j}(x,\cdot )$ satisfy the ($%
\triangle _{2}$)-condition in the sense of Definition \ref{N-func}.
\end{assumption}

It follows from Assumption \ref{asump-52} that there exist two constants $%
c_{1},c_{2}\in (0,1]$ such that for $\lambda _{N}$-a.e. $x\in \Omega $ if $%
j=1$ and $\sigma $-a.e. $x\in \partial \Omega $ if $j=2$ and for all $t\in 
\mathbb{R}$,%
\begin{equation}
c_{j}t\alpha _{j}(x,t)\leq \Lambda _{j}(x,t)\leq t\alpha _{j}(x,t).
\label{delta-22}
\end{equation}

Next, let 
\begin{equation*}
L_{\Lambda _{1}}(\Omega ):=\left\{ u:\Omega \rightarrow \mathbb{R}\;%
\mbox{
measurable: }\;\int_{\Omega }\Lambda _{1}(x,u)dx<\infty \right\}
\end{equation*}%
and 
\begin{equation*}
L_{\Lambda _{2}}(\partial \Omega ):=\left\{ u:\partial \Omega \rightarrow 
\mathbb{R}\;\mbox{ measurable: }\;\int_{\partial \Omega }\Lambda
_{2}(x,u)d\sigma <\infty \right\} .
\end{equation*}%
Since $\Lambda _{j}(x,\cdot )$ and $\widetilde{\Lambda }_{j}(x,\cdot )$
satisfy the $(\triangle _{2})$-condition, it follows from \cite[Theorem 8.19]%
{Adam}, that $L_{\Lambda _{1}}(\Omega )$ and $L_{\Lambda _{2}}(\partial
\Omega ),$ endowed respectively with the norms 
\begin{equation*}
\Vert u\Vert _{\Lambda _{1},\Omega }:=\inf \left\{ k>0:\int_{\Omega }\Lambda
_{1}\left( x,\frac{u(x)}{k}\right) dx\leq 1\right\} ,
\end{equation*}%
and 
\begin{equation*}
\Vert u\Vert _{\Lambda _{2},\partial \Omega }:=\inf \left\{
k>0:\int_{\partial \Omega }\Lambda _{2}\left( x,\frac{u(x)}{k}\right)
d\sigma \leq 1\right\} ,
\end{equation*}%
are reflexive Banach spaces. Moreover, by \cite[Section 8.11, p.234]{Adam},
the following generalized versions of H\"{o}lder's inequality will also
become useful in the sequel,%
\begin{equation}
\left\vert \int_{\Omega }uvdx\right\vert \leq 2\Vert u\Vert _{\Lambda
_{1},\Omega }\Vert v\Vert _{\widetilde{\Lambda }_{1},\Omega }  \label{hold-1}
\end{equation}%
and 
\begin{equation}
\left\vert \int_{\partial \Omega }uv\;d\sigma\right\vert \leq 2\Vert u\Vert
_{\Lambda _{2},\partial \Omega }\Vert v\Vert _{\widetilde{\Lambda }%
_{2},\partial \Omega }.  \label{hold-2}
\end{equation}

\subsection{Existence and uniqueness of weak solutions of perturbed equations%
}

Let 
\begin{equation*}
\mathcal{V}:=\{U:=(u,u|_{\partial \Omega }):\;u\in W^{1,p}(\Omega )\cap
L_{\Lambda _{1}}(\Omega ),\;u_{\mid \partial \Omega }\in W^{1,q}(\partial
\Omega )\cap L_{\Lambda _{2}}(\partial \Omega )\}.
\end{equation*}%
Then for every $1<p,q<\infty $, $\mathcal{V}$ endowed with the norm%
\begin{equation*}
\Vert U\Vert _{\mathcal{V}}=\Vert u\Vert _{W^{1,p}(\Omega )}+\Vert u\Vert
_{\Lambda _{1},\Omega }+\Vert u\Vert _{W^{1,q}(\partial \Omega )}+\Vert
u\Vert _{\Lambda _{2},\partial \Omega }
\end{equation*}%
is a reflexive Banach space. Recall that $\rho =1$. Throughout the
following, we denote by $\mathcal{V}^{\prime }$ the dual of $\mathcal{V}$. 
%If $(f,g)\in
%\mathcal{V}^{\prime },$ then (\ref{weak111}) makes sense.\newline

\begin{definition}
\label{def-form}A function $U=(u,u|_{\partial \Omega })\in \mathcal{V}$ is
said to be a weak solution of \eqref{eq-weak}, if for every $V\in \mathcal{V}%
=(v,v|_{\partial \Omega }),$%
\begin{equation}
\mathcal{A}(U,V)=\int_{\Omega }fvdx+\int_{\partial \Omega }gvd\sigma ,
\label{weak111}
\end{equation}%
provided that the integrals on the right-hand side exist. Here,%
\begin{align*}
\mathcal{A}(U,V)& :=\int_{\Omega }|\nabla u|^{p-2}\nabla u\cdot \nabla
vdx+\int_{\Omega }|u|^{p-2}uvdx \\
& +\int_{\Omega }\alpha _{1}(x,u)vdx+\int_{\partial \Omega }|\nabla _{\Gamma
}u|^{q-2}\nabla _{\Gamma }u\cdot \nabla _{\Gamma }vd\sigma \\
& +\int_{\partial \Omega }|u|^{q-2}uvd\sigma +\int_{\partial \Omega }\alpha
_{2}(x,u)vd\sigma .
\end{align*}
\end{definition}

\begin{lemma}
\label{lem-hemi-2} Assume Assumptions \ref{asump-51} and \ref{asump-52}. Let 
$1<p,q<\infty $ and $U\in \mathcal{V}$ be fixed. Then the functional $%
V\mapsto \mathcal{A}(U,V)$ belongs to $\mathcal{V}^{\prime }$. Moreover, $%
\mathcal{A}$ is strictly monotone, hemicontinuous and coercive.
\end{lemma}

\begin{proof}
Let $U=(u,u|_{\partial \Omega })\in \mathcal{V}$ be fixed. It is clear that $%
\mathcal{A}(U,\cdot )$ is linear. Let $V=(v,v|_{\partial \Omega })\in 
\mathcal{V}$. Then, exploiting \eqref{hold-1} and \eqref{hold-2}, we obtain%
\begin{align}  \label{bounded}
& \left\vert \mathcal{A}(U,V)\right\vert \leq \Vert u\Vert
_{W^{1,p}(\Omega)}^{p-1}\|v\|_{W^{1,p}(\Omega)}+\Vert u\Vert
_{W^{1,q}(\partial \Omega )}^{q-1}\Vert v\Vert _{W^{1,q}(\partial \Omega )}
\\
& +2\max \left\{ 1,\int_{\Omega }\widetilde{\Lambda }_{1}(x,\alpha
_{1}(x,u))\;dx\right\} \Vert v\Vert _{\Lambda _{1},\Omega }  \notag \\
& +2\max \left\{ 1,\int_{\partial \Omega }\widetilde{\Lambda }_{2}(x,\alpha
_{2}(x,u))\;d\sigma \right\} \Vert v\Vert _{\Lambda _{2},\partial \Omega } 
\notag \\
& \leq K(U)\Vert V\Vert _{\mathcal{V}},  \notag
\end{align}%
where%
\begin{align*}
K(U):=& \Vert u\Vert _{W^{1,p}(\Omega )}^{p-1}+2\max \left\{ 1,\int_{\Omega }%
\widetilde{\Lambda }_{1}(x,\alpha _{1}(x,u))\;dx\right\} \\
& +\Vert u\Vert _{W^{1,q}(\partial \Omega )}^{q-1}+2\max \left\{
1,\int_{\partial \Omega }\widetilde{\Lambda }_{2}(x,\alpha
_{2}(x,u))\;d\sigma \right\} .
\end{align*}%
This shows $\mathcal{A}(U,\cdot )\in \mathcal{V}^{\prime },$ for every $U\in 
\mathcal{V}$.

Next, let $U,V\in \mathcal{V}$. Then, using \eqref{in-ab} and the fact that $%
\alpha _{j}(x,\cdot )$ are monotone nondecreasing, that is, $(\alpha
_{j}(x,t)-\alpha _{j}(x,s))(t-s)\geq 0,$ for all $t,s\in \mathbb{R},$ we
obtain%
\begin{align}
& \mathcal{A}(U,U-V)-\mathcal{A}(V,U-V)  \label{strictt} \\
& =\int_{\Omega }\left( |\nabla u|^{p-2}\nabla u-|\nabla v|^{p-2}\nabla
v\right) \cdot \nabla (u-v)dx+\int_{\Omega }\left(
|u|^{p-2}u-|v|^{p-2}v\right) (u-v)dx  \notag \\
& +\int_{\Omega }\left( \alpha _{1}(x,u)-\alpha _{1}(x,v)\right)
(u-v)dx+\int_{\partial \Omega }\left( |u|^{q-2}u-|v|^{q-2}v\right)
(u-v)d\sigma  \notag \\
& +\int_{\partial \Omega }\left( |\nabla _{\Gamma }u|^{q-2}\nabla _{\Gamma
}u-|\nabla _{\Gamma }v|^{q-2}\nabla v\right) \cdot \nabla _{\Gamma
}(u-v)d\sigma  \notag \\
& +\int_{\partial \Omega }\left( \alpha _{2}(x,u)-\alpha _{1}(x,v)\right)
(u-v)d\sigma  \notag \\
& \geq \int_{\Omega }\left( |\nabla u|+|\nabla v|\right) ^{p-2}|\nabla
(u-v)|^{2}dx+\int_{\Omega }\left( |u|+|v|\right) ^{p-2}|u-v|^{2}dx  \notag \\
& +\int_{\partial \Omega }\left( |\nabla _{\Gamma }u|+|\nabla _{\Gamma
}v|\right) ^{p-2}|\nabla _{\Gamma }(u-v)|^{2}d\sigma +\int_{\partial \Omega
}\left( |u|+|v|\right) ^{p-2}|u-v|^{2}d\sigma  \notag \\
& \geq 0.  \notag
\end{align}%
This shows that $\mathcal{A}$ is monotone. The estimate \eqref{strictt} also
shows that%
\begin{equation*}
\mathcal{A}(U,U-V)-\mathcal{A}(V,U-V)>0,
\end{equation*}%
for all $U,V\in V$ with $U\neq V$, that is, $u\ne v$ or $u|_{\partial\Omega}%
\ne v|_{\partial\Omega}$. Thus, $\mathcal{A}$ is strictly monotone.

The continuity of the norm function and the continuity of $\alpha
_{j}(x,\cdot ),$ $j=1,2$ imply that $\mathcal{A}$ is hemicontinuous.

Finally, since $\Lambda _{j}$ and $\widetilde{\Lambda }_{j}$ satisfy the $%
(\triangle _{2}^{0})$-condition, from Proposition \ref{prop:coercive} and
Corollary \ref{cor:coercive}, it follows%
\begin{align*}
\lim_{\Vert u\Vert _{\Lambda _{1},\Omega }\rightarrow +\infty }\frac{%
\int_{\Omega }u\alpha _{1}(x,u)\;dx}{\Vert u\Vert _{\Lambda _{1},\Omega }}
=+\infty , \mbox{ and }\; \lim_{\Vert u\Vert _{\Lambda _{2},\partial \Omega
}\rightarrow +\infty }\frac{\int_{\partial \Omega }u\alpha
_{2}(x,u)\;d\sigma }{\Vert u\Vert _{\Lambda _{2},\partial \Omega }}&
=+\infty .
\end{align*}%
Consequently, we deduce%
\begin{equation}
\lim_{\Vert U\Vert _{\mathcal{V}}\rightarrow +\infty }\frac{\mathcal{A}(U,U)%
}{\Vert U\Vert _{\mathcal{V}}}=+\infty ,  \label{coerc}
\end{equation}%
which shows that $\mathcal{A}$ is coercive. The proof of the lemma is
finished.
\end{proof}

The following result is concerned with the existence and uniqueness of weak
solutions to problem \eqref{eq-weak}.

\begin{theorem}
\label{sol-1} Assume Assumptions \ref{asump-51} and \ref{asump-52}. Let $%
1<p,q<\infty $, $p_{1}\geq p\ast $ and $q_{1}\geq q\ast $, where $p\ast
:=p/(p-1)$ and $q\ast :=q/(q-1)$. Then for every $(f,g)\in X^{p_{1},q_{1}}(%
\overline{\Omega },\mu )$, there exists a unique function $U\in \mathcal{V}$
which is a weak solution to \eqref{eq-weak}.
\end{theorem}

\begin{proof}
Let $\langle \cdot ,\cdot \rangle $ denote the duality between $\mathcal{V}$
and $\mathcal{V}^{\prime }$. Then, from Lemma \ref{lem-hemi-2}, it follows
that for each $U\in \mathcal{V}$, there exists $A(U)\in \mathcal{V}^{\prime
} $ such that 
\begin{equation*}
\mathcal{A}(U,V)=\langle A(U),V\rangle ,
\end{equation*}%
for every $V\in \mathcal{V}$. Hence, this relation defines an operator $A:\;%
\mathcal{V}\rightarrow \mathcal{V}^{\prime },$ which is bounded by (\ref%
{bounded}). Exploiting Lemma \ref{lem-hemi-2} once again, it is easy to see
that $A$ is monotone and coercive. It follows from Brodwer's theorem (see,
e.g., \cite[Theorem 5.3.22]{DM}), that $A(\mathcal{V})=\mathcal{V}^{\prime }$%
. Therefore, for every $F\in \mathcal{V}^{\prime }$ there exists $U\in 
\mathcal{V}$ such that $A(U)=F$, that is, for every $V\in \mathcal{V}$,%
\begin{equation*}
\langle A(U),V\rangle =\mathcal{A}(U,V)=\langle V,F\rangle .
\end{equation*}%
Since $W^{1,p}(\Omega )\hookrightarrow L^{p}(\Omega )$ and $W^{1,q}(\partial
\Omega )\hookrightarrow L^{q}(\partial \Omega )$ with dense injection, by
duality, we have $X^{p\ast ,q\ast }(\overline{\Omega },\mu )\hookrightarrow 
\mathcal{V}^{\prime }$. Since $\Omega $ is bounded and $\sigma (\partial
\Omega )<\infty $, we obtain that%
\begin{equation*}
X^{p_{1},q_{1}}(\overline{\Omega },\mu )\hookrightarrow X^{p\ast ,q\ast }(%
\overline{\Omega },\mu )\hookrightarrow \mathcal{V}^{\prime }.
\end{equation*}%
This shows the existence of weak solutions. The uniqueness follows from the
fact that $\mathcal{A}$ is strictly monotone (cf. Lemma \ref{lem-hemi-2}).
This completes the proof of the theorem.
\end{proof}

\begin{corollary}
Let the assumptions of Theorem \ref{sol-1} be satisfied. Let%
\begin{equation}
p_{h}:=\frac{Np}{N(p-1)+p},\;\;q_{h}:=\frac{p(N-1)}{N(p-1)},\;\mbox{ and }%
q_{k}:=\frac{q(N-1)}{N(q-1)+1}.  \label{p-q-h}
\end{equation}

\begin{enumerate}
\item Let $1<p<N$, $1<q<p(N-1)/N$, $p_{1}\geq p_{h}$ and $q_{1}\geq q_{k}$.
Then for every $(f,g)\in X^{p_{1},q_{1}}(\Omega ,\mu )$, there exists a
function $U\in \mathcal{V}$ which is the unique weak solution to %
\eqref{eq-weak}.

\item Let $1<q<N-1$, $1<p<Nq/(N-1)$, $p_{1}\geq p_{h}$ and $q_{1}\geq q_{h}$%
. Then for every $(f,g)\in X^{p_{1},q_{1}}(\Omega ,\mu )$, there exists a
function $U\in \mathcal{V}$ which is the unique weak solution to %
\eqref{eq-weak}.
\end{enumerate}
\end{corollary}

\begin{proof}
We first prove (1). Let $1<p<N$ and $1<q<p(N-1)/N$ and let $p_{1}\geq p_{h}$
and $q_{1}\geq q_{k},$ where $p_{h}$ and $q_{k}$ are given by \eqref{p-q-h}.
Let $p_{s}:=Np/(N-p)$ and $q_{t}:=(N-1)q/(N-1-q)$. Since $W^{1,p}(\Omega
)\hookrightarrow L^{p_{s}}(\Omega )$ and $W^{1,q}(\partial \Omega
)\hookrightarrow L^{q_{t}}(\partial \Omega )$ with dense injection, then by
duality, %\begin{equation*}
$X^{p_{h},q_{k}}(\overline{\Omega },\mu )\hookrightarrow \mathcal{V}^{\prime
}$, %\end{equation*}%
where $\displaystyle1/p_{s}+1/p_{h}=1$ and $\displaystyle1/q_{t}+1/q_{k}=1$.
Since $\mu (\overline{\Omega })<\infty $, we have that 
\begin{equation*}
X^{p_{1},q_{1}}(\overline{\Omega },\mu )\hookrightarrow X^{p_{h},q_{h}}(%
\overline{\Omega },\mu )\hookrightarrow \mathcal{V}^{\prime }.
\end{equation*}%
Hence, for every $F:=(f,g)\in X^{p_{1},q_{1}}(\overline{\Omega },\mu
)\hookrightarrow \mathcal{V}^{\prime }$, there exists $U\in \mathcal{V}$
such that for every $V\in \mathcal{V}$, 
\begin{equation*}
\langle A(U),V\rangle =\mathcal{A}(U,V)=\int_{\Omega }fv\;dx+\int_{\partial
\Omega }gv\;d\sigma .
\end{equation*}%
The uniqueness of the weak solution follows again from the fact that $%
\mathcal{A}$ is strictly monotone.

In order to prove the second part, we use the the embeddings $W^{1,p}(\Omega
)\hookrightarrow L^{p_{s}}(\Omega )$, $W^{1,p}(\Omega )\hookrightarrow
L^{q_{s}}(\partial \Omega )$ and proceed exactly as above. We omit the
details.
\end{proof}

\subsection{Properties of the solution operator of the perturbed equation}

In the sequel, we establish some interesting properties of the solution
operator $A$ to problem \eqref{eq-weak}. We begin by assuming the following.

\begin{assumption}
\label{assump}Suppose that $\alpha _{j},$ $j=1,2,$ satisfy the following
conditions:%
\begin{equation}
\begin{cases}
\displaystyle\mbox{there are constants }\;c_{j}\in (0,1]\;\mbox{ such that }
\\ 
\displaystyle c_{j}\left\vert \alpha _{j}(x,\xi -\eta )\right\vert \leq
\left\vert \alpha _{j}(x,\xi )-\alpha _{j}(x,\eta )\right\vert \;%
\mbox{ for
all }\;\xi ,\eta \in \mathbb{R}.%
\end{cases}
\label{G}
\end{equation}
\end{assumption}

\begin{theorem}
\label{cor-existence} Assume Assumptions \ref{asump-51}, \ref{asump-52} and %
\ref{assump}. Let $p,q\geq 2$ and let $A:\;\mathcal{V}\rightarrow \mathcal{V}%
^{\prime }$ be the continuous and bounded operator constructed in the proof
of Theorem \ref{sol-1}. Then $A$ is injective and hence, invertible and its
inverse $A^{-1}$ is also continuous and bounded.
\end{theorem}

\begin{proof}
First, we remark that, since%
\begin{equation*}
\left( \alpha _{j}(x,t)-\alpha _{j}(x,s)\right) (t-s)\geq 0\text{, for all }%
t,s\in \mathbb{R},
\end{equation*}%
for $\lambda _{N}$-a.e.$x\in \Omega $ if $j=1$ and $\sigma $-a.e. $x\in
\partial \Omega $ if $j=2$, it follows from (\ref{G}) that, for all $t,s\in 
\mathbb{R}$, 
\begin{equation}
\left( \alpha _{j}(x,t)-\alpha _{j}(x,s)\right) (t-s)\geq c_{j}\alpha
_{j}(x,t-s)\cdot (t-s).  \label{cond-G}
\end{equation}%
Let $U,V\in \mathcal{V}$ and $p,q\in \lbrack 2,\infty )$. Then, exploiting %
\eqref{ine-BW}, \eqref{cond-G} and the ($\triangle _{2}$)-condition, we
obtain%
\begin{align}  \label{cool}
&\langle A(U)-A(V),U-V\rangle =\mathcal{A}(U,U-V)-\mathcal{A}(V,U-V) \\
& =\int_{\Omega }\left( |\nabla u|^{p-2}\nabla u-|\nabla v|^{p-2}\nabla
v\right) \cdot \nabla (u-v)dx +\int_{\Omega }\left(
|u|^{p-2}u-|v|^{p-2}v\right) (u-v)dx  \notag \\
& +\int_{\Omega }\left( \alpha _{1}(x,u)-\alpha _{1}(x,v)\right) (u-v)dx
+\int_{\partial \Omega }\left( |\nabla _{\Gamma }u|^{q-2}\nabla _{\Gamma
}u-|\nabla _{\Gamma }v|^{q-2}\nabla _{\Gamma }v\right) \cdot \nabla _{\Gamma
}(u-v)d\sigma  \notag \\
& +\int_{\partial \Omega }\left( |u|^{q-2}u-|v|^{q-2}v\right) (u-v)d\sigma
+\int_{\partial \Omega }\left( \alpha _{2}(x,u)-\alpha _{2}(x,v)\right)
(u-v)d\sigma  \notag \\
& \geq \Vert u-v\Vert _{W^{1,p}(\Omega )}^{p}+c_{1}\int_{\Omega }\Lambda
_{1}(x,u-v)dx +\Vert u-v\Vert _{W^{1,q}(\partial \Omega
)}^{q}+c_{2}\int_{\partial \Omega }\Lambda _{2}(x,u-v)d\sigma .  \notag
\end{align}%
This implies that $\langle A(U)-A(V),U-V\rangle >0,$ for all $U,V\in 
\mathcal{V}$ with $U\neq V$ (that is, $u\neq v$, or $u|_{\partial \Omega
}\neq v|_{\partial \Omega }$). Therefore, the operator $A$ is injective and
hence, $A^{-1}$ exists. Since for every $U\in \mathcal{V}$, 
\begin{equation*}
\mathcal{A}(U,U)=\langle A(U),U\rangle \leq \Vert A(U)\Vert _{\mathcal{V}%
^{\prime }}\Vert U\Vert _{\mathcal{V}},
\end{equation*}%
from the coercivity of $\mathcal{A}$ (see (\ref{coerc})), it is not
difficult to see that%
\begin{equation}
\lim_{\Vert U\Vert _{\mathcal{V}}\rightarrow +\infty }\Vert A(U)\Vert _{%
\mathcal{V}^{\prime }}=+\infty .  \label{in-coe}
\end{equation}%
Thus, $A^{-1}:\;\mathcal{V}^{\prime }\rightarrow \mathcal{V}$ is bounded.

Next, we show that $A^{-1}:\;\mathcal{V}^{\prime }\rightarrow \mathcal{V}$
is continuous. Assume that $A^{-1}$ is not continuous. Then there is a
sequence $F_{n}\in \mathcal{V}^{\prime }$ with $F_{n}\rightarrow F$ in $%
\mathcal{V}^{\prime }$ and a constant $\delta >0$ such that 
\begin{equation}
\Vert A^{-1}(F_{n})-A^{-1}(F)\Vert _{\mathcal{V}}\geq \delta ,
\label{in-cont}
\end{equation}%
for all $n\in \mathbb{N}$. Let $U_{n}:=A^{-1}(F_{n})$ and $U=A^{-1}(F)$.
Since $\left\{ F_{n}\right\} $ is a bounded sequence and $A^{-1}$ is
bounded, we have that $\left\{ U_{n}\right\} $ is bounded in $\mathcal{V}$.
Thus, we can select a subsequence, which we still denote by $\left\{
U_{n}\right\} ,$ which converges weakly to some function $V\in \mathcal{V}$.
Since $A(U_{n})-A(V)\rightarrow F-A(V)$ strongly in $\mathcal{V}$ and $%
U_{n}-V$ converges weakly to zero in $\mathcal{V}$, we deduce%
\begin{equation}
\lim_{n\rightarrow \infty }\langle A(U_{n})-A(V),U_{n}-V\rangle =0.
\label{zero}
\end{equation}%
From \eqref{cool} and \eqref{zero}, it follows that 
\begin{equation*}
\lim_{n\rightarrow \infty }\Vert u_{n}-v\Vert _{W^{1,p}(\Omega )}=0\;%
\mbox{and }\;\lim_{n\rightarrow \infty }\int_{\Omega }\Lambda
_{1}(x,u_{n}-v)dx=0,
\end{equation*}%
while%
\begin{equation*}
\lim_{n\rightarrow \infty }\Vert u_{n}-v\Vert _{W^{1,q}(\partial \Omega
)}=0\;\mbox{ and }\;\lim_{n\rightarrow \infty }\int_{\partial \Omega
}\Lambda _{2}(x,u_{n}-v)d\sigma =0.
\end{equation*}%
Therefore, $U_{n}\rightarrow V$ strongly in $\mathcal{V}$. Since $A$ is
continuous and%
\begin{equation*}
F_{n}=A(U_{n})\rightarrow A(V)=F=A(U)
\end{equation*}%
it follows from the injectivity of $A,$ that $U=V$. This shows that 
\begin{equation*}
\lim_{n\rightarrow \infty }\Vert A^{-1}(F_{n})-A^{-1}(F)\Vert _{\mathcal{V}%
}=\lim_{n\rightarrow \infty }\Vert U_{n}-U\Vert _{\mathcal{V}}=0,
\end{equation*}%
which contradicts \eqref{in-cont}. Hence, $A^{-1}:\;\mathcal{V}^{\prime
}\rightarrow \mathcal{V}$ is continuous. The proof is finished.
\end{proof}

\begin{corollary}
Let the assumptions of Theorem \ref{cor-existence} be satisfied. Let $%
p_{h},q_{h}$ and $q_{k}$ be as in \eqref{p-q-h} and let $A:\;\mathcal{V}%
\rightarrow \mathcal{V}^{\prime }$ be the continuous and bounded operator
constructed in the proof of Theorem \ref{sol-1}.

\begin{enumerate}
\item If $2\leq p<N$, $2\leq q<p(N-1)/N$, $p_{1}\geq p_{h}$ and $q_{1}\geq
q_{k}$, then $A^{-1}:\;X^{p_{1},q_{1}}(\overline{\Omega },\mu )\rightarrow
X^{p_{s},q_{t}}(\overline{\Omega },\mu )$ is continuous and bounded.
Moreover, $A^{-1}:\;X^{p_{1},q_{1}}(\overline{\Omega },\mu )\rightarrow 
\mathcal{V}\cap X^{r,s}(\overline{\Omega },\mu )$ is compact for every $r\in
(1,p_s)$ and $s\in (1,q_s)$.

\item If $2\leq q<N-1$, $2\leq p<qN/(N-1)$, $p_{1}\geq p_{h}$ and $q_{1}\geq
q_{h}$, then the operator $A^{-1}:\;X^{p_{1},q_{1}}(\overline{\Omega },\mu
)\rightarrow X^{p_{s},q_{s}}(\overline{\Omega },\mu )$ is continuous and
bounded. Moreover, $A^{-1}:\;X^{p_{1},q_{1}}(\overline{\Omega },\mu
)\rightarrow \mathcal{V}\cap X^{r,s}(\overline{\Omega },\mu )$ is compact
for every $r\in (1,p_s)$ and $s\in (1,q_s)$.
\end{enumerate}
\end{corollary}

\begin{proof}
We only prove the first part. The second part of the proof follows by
analogy and is left to the reader. Let $2\leq p<N$, $2\leq q<p(N-1)/N$, $%
p_{1}\geq p_{h}$ and $q_{1}\geq q_{k}$ and let $F\in X^{p_{1},q_{1}}(%
\overline{\Omega },\mu )$. Proceeding exactly as in the proof of Theorem \ref%
{cor-existence}, we obtain%
\begin{equation*}
\Vert A^{-1}(F)\Vert _{p_{s},q_{t}}\leq C_{1}\Vert A^{-1}(F)\Vert _{\mathcal{%
V}}\leq C\Vert F\Vert _{{\mathcal{V}^{\prime }}}\leq C_{2}\Vert F\Vert
_{p_{1},q_{1}}.
\end{equation*}%
Hence, the operator $A^{-1}:\;X^{p_{1},q_{1}}(\overline{\Omega },\mu
)\rightarrow X^{p_{s},q_{t}}(\overline{\Omega },\mu )$ is bounded. Finally,
using the facts that $X^{p_{1},q_{1}}(\overline{\Omega },\mu
)\hookrightarrow \mathcal{V}^{\prime }$, $A^{-1}:\;\mathcal{V}^{\prime
}\rightarrow \mathcal{V}$ is continuous and $\mathcal{V}\hookrightarrow
X^{p_{s},q_{t}}(\overline{\Omega },\mu )$, we easily deduce that $%
A^{-1}:\;X^{p_{1},q_{1}}(\overline{\Omega },\mu )\rightarrow X^{p_{s},q_{t}}(%
\overline{\Omega },\mu )$ is continuous.

Now, let $1<r<p_s$ and $1<s<q_s$. Since the injection $\mathcal{V}%
\hookrightarrow X^{r,s}(\overline{\Omega },\mu)$ is compact, then by
duality, the injection $X^{r^{\prime},s^{\prime}}(\overline{\Omega }%
,\mu)\hookrightarrow (\mathcal{V})^*$ is compact for every $%
r^{\prime}>p_s^{\prime}=p_h$ and $s^{\prime}>q_s^{\prime}=q_h$. This,
together with the fact that $A^{-1}:\; (\mathcal{V})^*\to \mathcal{V}$ is
continuous and bounded, imply that $A^{-1}:\; X^{p_1,q_1}(\overline{\Omega }%
,\mu)\to \mathcal{V}$ is compact for every $p_1>p_h$ and $q_1>q_h$.

It remains to show that $A^{-1}$ is also compact as a map into $X^{r,s}(%
\overline{\Omega },\mu)$ for every $r\in (1,p_s)$ and $s\in (1,q_s)$. Since $%
A^{-1}$ is bounded, we have to show that the image of every bounded set $%
\mathcal{B}\subset \mathbb{X}^{p_1,q_1}(\Omega,\mu)$ is relatively compact
in $X^{r,s}(\overline{\Omega },\mu)$ for every $r\in (1,p_s)$ and $s\in
(1,q_s)$. Let $U_n$ be a sequence in $A^{-1}(\mathcal{B})$. Let $%
F_n=A(U_n)\in\mathcal{B}$. Since $\mathcal{B}$ is bounded, then the sequence 
$F_n$ is bounded. Since $A^{-1}$ is compact as a map into $\mathcal{V}$, it
follows that there is a subsequence $F_{n_k}$ such that $A^{-1}(F_{n_k})\to
U\in \mathcal{V}$. We may assume that $U_n=A^{-1}(F_{n})\to U$ in $\mathcal{V%
}$ and hence, in $X^{p,p}(\overline{\Omega },\mu)$. It remains to show that $%
U_n\to U$ in $X^{r,s}(\overline{\Omega },\mu)$. Let $r\in [p,p_s)$ and $s\in
[p,q_s)$. Since $U_n:=(u_n,u_n|_{\partial \Omega})$ is bounded in $%
X^{p_s,q_s}(\overline{\Omega },\mu)$, a standard interpolation inequality
shows that there exists $\tau\in (0,1)$ such that 
\begin{equation*}
\||U_n-U_m\||_{r,s}\le\||U_n-U_m\||_{p,p}^\tau\||U_n-U_m\||_{p_s,q_s}^{1-%
\tau}\le C\||U_n-U_m\||_{p,p}^\tau.
\end{equation*}
As $U_n$ converges in $X^{p,p}(\overline{\Omega },\mu)$, it follows from the
preceding inequality that $U_n$ is a Cauchy sequence in $X^{r,s}(\overline{%
\Omega },\mu)$ and therefore converges in $X^{r,s}(\overline{\Omega },\mu)$.
Hence, $A^{-1}:\; X^{p_1,q_1}(\overline{\Omega },\mu)\to \mathcal{V}\cap
X^{r,s}(\overline{\Omega },\mu)$ is compact for every $r\in [p,p_s)$ and $%
s\in [p,q_s)$. The case $r,s\in (1,p)$ follows from the fact that $X^{p,p}(%
\overline{\Omega },\mu)\hookrightarrow X^{r,s}(\overline{\Omega },\mu)$ and
the proof is finished
\end{proof}

\subsection{Statement and proof of the main result}

We will now establish under what conditions the operator $A^{-1}$ maps $%
X^{p_{1},q_{1}}(\overline{\Omega },\mu )$ boundedly and continuously into $%
X^{\infty }(\overline{\Omega },\mu )$. The following is the main result of
this section.

\begin{theorem}
\label{sol-bounded} Let the assumptions of Theorem \ref{cor-existence} be
satisfied.

\begin{enumerate}
\item Suppose $2\leq p<N$ and $2\leq q<\infty $. Let 
\begin{equation*}
p_{1}>\frac{p_{s}}{p_{s}-p}=\frac{N}{p}\;\mbox{ and }\;q_{1}>\frac{q_{s}}{%
q_{s}-p}=\frac{N-1}{p-1}.
\end{equation*}%
Let $f\in L^{p_{1}}(\Omega ),\;g\in L^{q_{1}}(\partial \Omega )$ and $U,V\in 
\mathcal{V}$ be such that for every function $\Phi =(\varphi ,\varphi
|_{\partial \Omega })\in \mathcal{V}$, 
\begin{equation}
\mathcal{A}(U,\Phi )-\mathcal{A}(V,\Phi )=\int_{\Omega }f\varphi
\;dx+\int_{\partial \Omega }g\varphi \;d\sigma .  \label{eq-main}
\end{equation}%
Then there is a constant $C=C(N,p,q,\Omega )>0$ such that 
\begin{equation*}
\Vert |U-V\Vert |_{\infty }^{p-1}\leq C(\Vert f\Vert _{p_{1},\Omega }+\Vert
g\Vert _{q_{1},\partial \Omega }).
\end{equation*}

\item Suppose $2\leq p=q<N-1$. Let 
\begin{equation*}
p_{1}>\frac{p_{s}}{p_{s}-p}=\frac{N}{p}\;\mbox{ and }\;q_{1}>\frac{p_{t}}{%
p_{t}-p}=\frac{N-1}{p}.
\end{equation*}%
Let $f\in L^{p_{1}}(\Omega ),\;g\in L^{q_{1}}(\partial \Omega )$ and $U,V\in 
\mathcal{V}$ satisfy \eqref{eq-main}. Then there is a constant $%
C=C(N,p,q,\Omega )>0$ such that 
\begin{equation*}
\Vert |U-V\Vert |_{\infty }^{p-1}\leq C(\Vert f\Vert _{p_{1},\Omega }+\Vert
g\Vert _{q_{1},\partial \Omega }).
\end{equation*}
\end{enumerate}
\end{theorem}

\begin{proof}
Let $U,V\in \mathcal{V}$ satisfy \eqref{eq-main}. Let $k\geq 0$ be a real
number and set 
\begin{equation*}
w_{k}:=(|u-v|-k)^{+}\func{sgn}(u-v)\; W_k:=(w_k,w_k|_{\partial\Omega}) \;%
\mbox{ and }\;w:=|u-v|.
\end{equation*}%
Let $A_{k}:=\{x\in \overline{\Omega }:|w(x)|\geq k\}$, and $%
A_{k}^{+}:=\{x\in \overline{\Omega }:w(x)\geq k\},\;\;A_{k}^{-}:=\{x\in 
\overline{\Omega }:w(x)\leq -k\}$. Clearly $W_k\in\mathcal{V}$ and $%
A_{k}=A_{k}^{+}\cup A_{k}^{-}$. We claim that there exists a constant $C>0$
such that%
\begin{equation}
C\mathcal{A}(W_{k},W_{k})\leq \mathcal{A}(U,W_{k})-\mathcal{A}(V,W_{k}),
\label{calc*}
\end{equation}%
for all $U,V\in \mathcal{V}$. Using the definition of the form $\mathcal{A}$%
, we have%
\begin{align}
& \mathcal{A}(U,W_{k})-\mathcal{A}(V,W_{k})  \label{calc1} \\
& =\int_{\Omega }(|\nabla u|^{p-2}\nabla u-|\nabla v|^{p-2}\nabla v)\cdot
\nabla w_{k}dx+\int_{\Omega }(|u|^{p-2}u-|v|^{p-2}v)w_{k}dx  \notag \\
& +\int_{\Omega }(\alpha _{1}(x,u)-\alpha _{2}(x,v))w_{k}dx+\int_{\partial
\Omega }(|u|^{q-2}u-|v|^{q-2}v)w_{k}d\sigma  \notag \\
& +\int_{\partial \Omega }(|\nabla _{\Gamma }u|^{p-2}\nabla _{\Gamma
}u-|\nabla _{\Gamma }v|^{p-2}\nabla _{\Gamma }v)\cdot \nabla _{\Gamma
}w_{k}d\sigma +\int_{\partial \Omega }(\alpha _{2}(x,u)-\alpha
_{2}(x,v))w_{k}d\sigma .  \notag
\end{align}%
Since%
%\begin{equation*}
$\nabla w_{k}=%
\begin{cases}
\nabla (u-v) & \mbox{ in }A(k), \\ 
0 & \mbox{ otherwise, }%
\end{cases}%
$ %\end{equation*}%
we can rewrite (\ref{calc1}) as follows:%
\begin{align}  \label{calc2}
& \mathcal{A}(U,W_{k})-\mathcal{A}(V,W_{k}) =\int_{A(k)\cap \Omega }(|\nabla
u|^{p-2}\nabla u-|\nabla v|^{p-2}\nabla v)\cdot \nabla (u-v)dx \\
&{}+\int_{A(k)\cap \partial \Omega }(|\nabla _{\Gamma }u|^{q-2}\nabla
_{\Gamma }u-|\nabla _{\Gamma }v|^{q-2}\nabla _{\Gamma }v)\cdot \nabla
_{\Gamma }(u-v)d\sigma  \notag \\
& {}+\lambda \int_{A(k)\cap \Omega }(|u|^{p-2}u-|v|^{p-2}v)w_{k}dx
+\int_{A(k)\cap \Omega }(\alpha _{1}(x,u)-\alpha _{1}(x,v))w_{k}dx  \notag \\
&{}+\int_{A(k)\cap \partial \Omega }(\alpha _{2}(x,u)-\alpha
_{2}(x,v))w_{k}d\sigma .  \notag
\end{align}%
Exploiting inequality \eqref{ine-BW}, from (\ref{calc2}) and (\ref{cond-G}),
we deduce%
\begin{align}
& \mathcal{A}(U,W_{k})-\mathcal{A}(V,W_{k})  \label{calc3} \\
& \geq \int_{A(k)\cap \Omega }\left( |\nabla w_{k}|^{p}+|w_{k}|^{p}\right)
dx+\int_{A(k)\cap \Omega }c_{1}\alpha _{1}(x,w_{k})w_{k}dx  \notag \\
& +\int_{A(k)\cap \Omega }(|u|^{p-2}uw_{k}-|v|^{p-2}vw_{k}-|w_{k}|^{p})dx 
\notag \\
& +\int_{A(k)\cap \Omega }(\alpha _{1}(x,u)-\alpha _{1}(x,v)-c_{1}\alpha
_{1}(x,w_{k}))w_{k}dx  \notag \\
& +\int_{A(k)\cap \partial \Omega }\left( |\nabla _{\Gamma
}w_{k}|^{q}+|w_{k}|^{q}\right) d\sigma +\int_{A(k)\cap \partial \Omega
}c_{2}\alpha _{2}(x,w_{k})w_{k}d\sigma  \notag \\
& +\int_{A(k)\cap \partial \Omega
}(|u|^{q-2}uw_{k}-|v|^{q-2}vw_{k}-|w_{k}|^{q})d\sigma  \notag \\
& +\int_{A(k)\cap \partial \Omega }(\alpha _{2}(x,u)-\alpha
_{2}(x,v)-c_{2}\alpha _{2}(x,w_{k}))w_{k}d\sigma  \notag \\
& \geq C\mathcal{A}(W_{k},W_{k}) +\int_{A(k)\cap \Omega
}(|u|^{p-2}uw_{k}-|v|^{p-2}vw_{k}-|w_{k}|^{p})dx  \notag \\
&{} +\int_{A(k)\cap \Omega }(\alpha _{1}(x,u)-\alpha _{1}(x,v)-c_{1}\alpha
_{1}(x,w_{k}))w_{k}dx  \notag \\
&{} +\int_{A(k)\cap \partial \Omega
}(|u|^{q-2}uw_{k}-|v|^{q-2}vw_{k}-|w_{k}|^{q})d\sigma  \notag \\
&{} +\int_{A(k)\cap \partial \Omega }(\alpha _{2}(x,u)-\alpha
_{2}(x,v)-c_{2}\alpha _{2}(x,w_{k}))w_{k}d\sigma ,  \notag
\end{align}%
where $c_{1},c_{2}$ are the constants from (\ref{cond-G}). Using (\ref{G})\
and the fact that $\alpha _{j}(x,\cdot )$ are strictly increasing, for $x\in
A_{k}^{+}$, we have%
\begin{align*}
c_{j}\alpha _{j}(x,w_{k}(x))& =c_{j}\alpha _{j}(x,u(x)-v(x)-k)\leq
c_{j}\alpha _{j}(x,u(x)-v(x)) \\
& \leq \alpha _{j}(x,u(x))-\alpha _{j}(x,v(x)).
\end{align*}%
Multiplying this inequality by $w_{k}(x)\geq 0,$ $x\in A_{k}^{+},$ yields%
\begin{equation}
(\alpha _{j}(x,u(x))-\alpha _{j}(x,v(x))-c_{j}\alpha
_{j}(x,w_{k}(x)))w_{k}(x)\geq 0.  \label{eq-pos1}
\end{equation}%
Similarly, for $x\in A_{k}^{-},$%
\begin{align*}
c_{j}\alpha _{j}(x,w_{k}(x))& =c_{j}\alpha _{j}(x,u(x)-v(x)+k)\geq
c_{j}\alpha _{j}(x,u(x)-v(x)) \\
& \geq \alpha _{j}(x,u(x))-\alpha _{j}(x,v(x)).
\end{align*}%
Hence, multiplying this inequality by $w_{k}(x)\leq 0$, we get%
\begin{equation}
(\alpha _{j}(x,u(x))-\alpha _{j}(x,v(x))-c_{j}\alpha
_{j}(x,w_{k}(x)))w_{k}(x)\geq 0,  \label{eq-pos2}
\end{equation}%
for all $x\in A_{k}^{-}$. Hence, on account of \eqref{eq-pos1} and %
\eqref{eq-pos2}, from (\ref{calc3}) we obtain the required estimate of (\ref%
{calc*}).

(a) To prove this part, note that from Definition \ref{def-form} it is clear
that, 
\begin{equation}
\Vert w_{k}\Vert _{W^{1,p}(\Omega )}^{p}\leq \mathcal{A}(W_{k},W_{k}).
\label{eq-norm}
\end{equation}%
Let $f\in L^{p_{1}}(\Omega )$ and $g\in L^{q_{1}}(\partial \Omega )$ with 
\begin{equation*}
p_{1}>\frac{p_{s}}{p_{s}-p}=\frac{N}{p}\;\mbox{ and }\;q_{1}>\frac{q_{s}}{%
q_{s}-p}=\frac{N-1}{p-1},
\end{equation*}%
and let $B\subset \overline{\Omega }$ be any $\mu $-measurable set. We claim
that there exists a constant $C\geq 0$ such that, for every $F\in
X^{p_{1},q_{1}}(\overline{\Omega },\mu )$ and $\varphi \in W^{1,p}(\Omega ),$
we have%
\begin{equation}
\Vert |F\varphi 1_{B}\Vert |_{1,1}\leq C\Vert |F\Vert |_{p_{1},q_{1}}\Vert
\varphi \Vert _{W^{1,p}(\Omega )}\Vert |\chi _{B}\Vert |_{p_{3},q_{3}},
\label{holder-ine-2}
\end{equation}%
where $p_{3}$ and $q_{3}$ are such that $1/p_{3}+1/p_{1}+1/p_{s}=1$ and $%
1/q_{3}+1/q_{1}+1/q_{s}=1$. In fact, note that if $n\in \mathbb{N}$ and $%
p_{i},$ $q_{i}\in \lbrack 1,\infty ],\;(i=1,\dots ,n)$ are such that%
\begin{equation*}
\sum_{i=1}^{n}\frac{1}{p_{i}}=\sum_{i=1}^{n}\frac{1}{q_{i}}=1,
\end{equation*}%
and, if $F_{i}\in X^{p_{i},q_{i}}(\overline{\Omega },\mu ),\;(i=1,\dots ,n)$%
, then by H\"{o}lder's inequality,%
\begin{equation}
\Vert |\prod_{i=1}^{n}F_{i}\Vert |_{1,1}\leq \prod_{i=1}^{n}\Vert
|F_{i}\Vert |_{p_{i},q_{i}}.  \label{holder-ine}
\end{equation}%
Since $W^{1,p}(\Omega )\hookrightarrow X^{p_{s},q_{s}}(\overline{\Omega }%
,\mu )$, \eqref{holder-ine-2} follows immediately from \eqref{holder-ine}
and the claim (\ref{holder-ine-2}) is proved. Next, it follows from (\ref%
{holder-ine-2}), that%
\begin{eqnarray*}
\int_{\overline{\Omega }}FW_{k}d\mu =\Vert |FW_{k}\Vert |_{1,1} &=&\Vert
|FW_{k}\chi _{A_{k}}\Vert |_{1,1} \\
&\leq &\Vert |F\Vert |_{p_{1},q_{1}}\Vert |w_{k}\Vert |_{W^{1,p}(\Omega
)}\Vert |\chi _{A_{k}}\Vert |_{p_{3},q_{3}},
\end{eqnarray*}%
where we recall that $1/p_{3}=\left( 1-1/p_{s}-1/p_{1}\right) >\left(
p-1\right) /p_{s}$ and $q_{3}<q_{s}/(p-1)$. Therefore, for every $k\geq 0$,%
\begin{equation*}
\mathcal{A}(U,W_{k})-\mathcal{A}(V,W_{k})\leq \Vert |F\Vert
|_{p_{1},q_{1}}\Vert w_{k}\Vert _{W^{1,p}(\Omega )}\Vert |\chi _{A_{k}}\Vert
|_{p_{3},q_{3}},
\end{equation*}%
which together with estimate (\ref{calc*})\ yields the desired inequality 
\begin{equation}
C\mathcal{A}(W_{k},W_{k})\leq \mathcal{A}(U,W_{k})-\mathcal{B}(V,W_{k})\leq
\Vert |F\Vert |_{p_{1},q_{1}}\Vert w_{k}\Vert _{W^{1,p}(\Omega )}\Vert |\chi
_{A_{k}}\Vert |_{p_{3},q_{3}},  \label{eq:fg-k-1}
\end{equation}%
It follows from \eqref{eq-norm} and \eqref{eq:fg-k-1}, that for every $k>0$, 
\begin{align*}
C\Vert w_{k}\Vert _{W^{1,p}(\Omega )}^{p}& \leq \mathcal{A}(W_{k},W_{k})\leq 
\mathcal{A}(U,W_{k})-\mathcal{A}(V,W_{k}) \\
& \leq \Vert |F\Vert |_{p_{1},q_{1}}\Vert w_{k}\Vert _{W^{1,p}(\Omega
)}\Vert |\chi _{A_{k}}|\Vert _{p_{3},q_{3}}.
\end{align*}%
Hence, for every $k>0$, $\Vert w_{k}\Vert _{W^{1,p}(\Omega )}^{p-1}\leq
C_{1}\Vert |\chi _{A_{k}}\Vert |_{p_{3},q_{3}}$. Using the fact $%
W^{1,p}(\Omega )\hookrightarrow X^{p_{s},q_{s}}(\overline{\Omega },\mu )$,
we obtain for every $k>0$, that%
\begin{equation*}
\Vert |w_{k}\Vert |_{p_{s},q_{s}}^{p-1}\leq C\Vert |F\Vert
|_{p_{1},q_{1}}\Vert |\chi _{A_{k}}\Vert |_{p_{3},q_{3}}.
\end{equation*}%
Let $h>k$. Then $A_{h}\subset A_{k}$ and on $A_{h}$ the inequality $%
|w_{k}|\geq \left( h-k\right) $ holds. Therefore, 
\begin{equation*}
\Vert |(h-k)\chi _{A_{h}}\Vert |_{p_{s},q_{s}}^{p-1}\leq C\Vert |F\Vert
|_{p_{1},q_{1}}\Vert |\chi _{A_{k}}\Vert |_{p_{3},q_{3}},
\end{equation*}%
which shows that%
\begin{equation}
\Vert |\chi _{A_{h}}\Vert |_{p_{s},q_{s}}^{p-1}\leq C\Vert |F\Vert
|_{p_{1},q_{1}}(h-k)^{-(p-1)}\Vert |\chi _{A_{k}}\Vert |_{p_{3},q_{3}}.
\label{333}
\end{equation}%
Let $C_{3}:=\Vert |1_{\overline{\Omega }}\Vert |_{p_{s},q_{s}}$, and%
\begin{equation*}
\delta :=\min \left\{ \frac{p_{s}}{p_{3}},\frac{q_{s}}{p_{3}}\right\}
>p-1,\;\delta _{0}:=\frac{\delta }{p-1}>1.
\end{equation*}%
Then 
\begin{align}
\Vert C_{3}^{-p_{s}/p_{3}}\chi _{A_{k}}\Vert _{\Omega ,p_{3}}& =\Vert
C_{3}^{-1}\chi _{A_{k}}\Vert _{\Omega ,p_{s}}^{p_{s}/p_{3}}\leq \Vert
C_{3}^{-1}\chi _{A_{k}}\Vert _{\Omega ,p_{s}}^{\delta }  \label{111} \\
& \leq \Vert |\chi _{A_{k}}\Vert |_{p_{s},q_{s}}^{\delta }C_{3}^{-\delta } 
\notag
\end{align}%
and 
\begin{align}
\Vert C_{3}^{-q_{s}/q_{3}}\chi _{A_{k}}\Vert _{\partial \Omega ,q_{3}}&
=\Vert C_{3}^{-1}\chi _{A_{k}}\Vert _{\partial \Omega
,q_{s}}^{q_{s}/q_{3}}\leq \Vert C_{3}^{-1}\chi _{A_{k}}\Vert _{\partial
\Omega ,q_{s}}^{\delta }  \label{222} \\
& \leq \Vert |\chi _{A_{k}}\Vert |_{p_{s},q_{s}}^{\delta }C_{3}^{-\delta }. 
\notag
\end{align}%
Choosing $C_{\Omega }:=C_{3}^{p_{s}/p_{3}-\delta }+C_{3}^{q_{s}/q_{3}-\delta
}$, from (\ref{111})-(\ref{222}) we have%
\begin{equation}
\Vert |\chi _{A_{k}}\Vert |_{p_{3},q_{3}}\leq C_{\Omega }\Vert |\chi
_{A_{k}}\Vert |_{p_{s},q_{s}}^{\delta }.  \label{444}
\end{equation}%
Therefore, combining (\ref{333}) with (\ref{444}), we get%
\begin{align}
\Vert |\chi _{A_{h}}\Vert |_{p_{s},q_{s}}^{p-1}& \leq C\Vert |F\Vert
|_{p_{1},q_{1}}(h-k)^{-(p-1)}\Vert |\chi _{A_{k}}\Vert
|_{p_{s},q_{s}}^{\delta }  \label{555} \\
& =C\Vert |F\Vert |_{p_{1},q_{1}}(h-k)^{-(p-1)}\left[ \Vert |\chi
_{A_{k}}\Vert |_{p_{s},q_{s}}^{p-1}\right] ^{\delta _{0}}.  \notag
\end{align}%
Setting $\psi (h):=\Vert |\chi _{A_{h}}\Vert |_{p_{s},q_{s}}^{p-1}$ in Lemma %
\ref{lem:fallend}, on account of (\ref{555}), we can find a constant $C_{2}$
(independent of $F$) such that%
\begin{equation*}
\Vert |\chi _{A_{K}}\Vert |_{p_{s},q_{s}}^{p-1}=0\;\text{ with }%
\;K:=C_{2}\Vert |F\Vert |_{p_{1},q_{1}}^{1/(p-1)}.
\end{equation*}%
This shows that $\mu (A_{K})=0,$ where $A_{K}=\{x\in \overline{\Omega }%
:|(u-v)(x)|\geq K\}$. Hence, we have $|u-v|\leq K,$ $\mu $-a.e. on $%
\overline{\Omega }$ so that 
\begin{equation*}
\Vert |U-V\Vert |_{\infty }^{p-1}\leq C_{2}\Vert |F\Vert
|_{p_{1},q_{1}}=C_{2}\left( \Vert f\Vert _{p_{1},\Omega }+\Vert g\Vert
_{q_{1},\partial \Omega }\right) ,
\end{equation*}%
which completes the proof of part (a).

(b) To prove this part, instead of \eqref{eq-norm} and \eqref{holder-ine-2},
one uses %\begin{equation*}
$\Vert W_{k}\Vert _{\mathcal{V}_{1}}^{p}\leq \mathcal{A}(W_{k},W_{k})$ 
%\end{equation*}
and %\begin{equation*}
$\Vert |F\varphi 1_{B}\Vert |_{1,1}\leq C\Vert |F\Vert |_{p_{1},q_{1}}\Vert
\varphi \Vert _{W^{1,p}(\Omega )}\Vert |\chi _{B}\Vert |_{p_{3},q_{3}}$, 
%\end{equation*}%
(where $p_{3}$ and $q_{3}$ are such that $1/p_{3}+1/p_{1}+1/p_{s}=1$ and $%
1/q_{3}+1/q_{1}+1/p_{t}=1$) and the embedding $\mathcal{V}\hookrightarrow 
\mathcal{V}_{1}\hookrightarrow X^{p_{s},p_{t}}(\overline{\Omega },\mu )$.
The remainder of the proof follows as in the proof of part (a).
\end{proof}

We conclude this section with the following example.

\begin{example}
Let $p\in \lbrack 2,\infty )$, $b:\partial \Omega \rightarrow (0,\infty )$
be a strictly positive and $\sigma $-measurable function and let 
\begin{equation*}
\beta (x,\xi ):=b(x)|\xi |^{p-2}\xi ,\quad \xi \in \mathbb{R}.
\end{equation*}%
Then, it is easy to verify that $\beta $ satisfies Assumptions \ref{asump-51}%
, \ref{asump-52} and \ref{assump} (see, e.g., \cite[Example 4.17]{BW}).
\end{example}

\end{document}